\documentclass[12pt, reqno]{amsart}
\usepackage{amsmath, amsthm, amscd, amsfonts, amssymb, graphicx, xcolor}
\usepackage[bookmarksnumbered, colorlinks, plainpages]{hyperref}
\usepackage{tikz}
\usetikzlibrary{arrows.meta, positioning}
\usepackage{tikz-cd}
\usepackage{enumitem}
\usepackage{xcolor}
\definecolor{RED}{named}{red}

\newtheorem{theorem}{Theorem}[section]
\newtheorem{lemma}[theorem]{Lemma}
\newtheorem{proposition}[theorem]{Proposition}
\newtheorem{corollary}[theorem]{Corollary}
\theoremstyle{definition}
\newtheorem{definition}[theorem]{Definition}
\newtheorem{example}[theorem]{Example}

\theoremstyle{remark}
\newtheorem{remark}[theorem]{Remark}
\numberwithin{equation}{section}

\theoremstyle{theorem}
\newtheorem*{theorem*}{Theorem}
\newtheorem*{proposition*}{Proposition}
\newtheorem*{corollary*}{Corollary}
\newtheorem*{lemma*}{Lemma}
\newtheorem*{example*}{Example}
\theoremstyle{remark}
\newtheorem*{remark*}{Remark}

\newcommand{\grouphull}[1]{\left\langle #1 \right\rangle}

\def\subfin{{\operatorname{Sub_{fin}}}}
\def\subfinnorm{{\operatorname{Sub_{fin}^{\triangleleft}}}}

\def\Eig{{\rm Eig}}
\newcommand{\envelopingOdometer}[1]{\mathcal{O}(#1)}

\newcommand{\upperhull}[1]{#1^{\mathfrak{u}}}
\newcommand{\conhull}[1]{#1^{\mathfrak{c}}}
\newcommand{\filterhull}[1]{#1^{\mathfrak{f}}}
\usepackage{stmaryrd}
\newcommand{\eigenhull}[1]{\left\llbracket #1 \right\rrbracket}

\newcommand{\allEigensets}[1]{\mathfrak{E}(#1)}
\newcommand{\allFilteredEigensets}[1]{\mathfrak{E}_\mathfrak{F}(#1)}

\begin{document}
\color{black}

\title
[Minimal Equicontinuous Actions on Stone Spaces]
{Minimal Equicontinuous Actions on Stone Spaces}

\author[]{Mar\'{\i}a Isabel Cortez}
\address{Facultad de Matem\'aticas, Pontificia Universidad Cat\'olica de Chile. Edificio Rolando Chuaqui, Campus San Joaquín. Avda. Vicuña Mackenna 4860, Macul, Chile.}
\email{maria.cortez@uc.cl}

\author[]{Till Hauser}
\address{Facultad de Matem\'aticas, Pontificia Universidad Cat\'olica de Chile. Edificio Rolando Chuaqui, Campus San Joaquín. Avda. Vicuña Mackenna 4860, Macul, Chile.}
\email{hauser.math@mail.de}

\thanks{This article was funded by the Deutsche Forschungsgemeinschaft (DFG, German Research Foundation) – 530703788.}

\begin{abstract}
    In this article we study minimal equicontinuous actions on Stone spaces, which we call \emph{subodometers}, and do neither assume that the space is metrizable, nor any assumptions on the acting group. We show that the set of eigenvalues is a complete invariant for subodometers. Furthermore, we characterize minimal rotations on Stone spaces, which we call \emph{odometers}, via the intersection stability of their sets of eigenvalues. We show that any non-empty family of odometers allows for a \emph{minimal common extension} and a \emph{maximal common factor}, that both are odometers and that they are unique up to conjugacy. We provide examples that a similar statement does not hold for subodometers. 
    
    We show that subodometers are given as inverse limits of minimal finite actions, that odometers are given as inverse limits of minimal finite rotations, and present how the minimal common extension and the maximal common factor of a non-empty family of odometers can be represented as an inverse limit. 
    
    We establish that a minimal action $X$ is a subodometer if and only if its Ellis semigroup $E(X)$ is an odometer, and present how an inverse limit representation of $E(X)$ can be derived from the representation of $X$. Furthermore, we establish the existence of a universal odometer that has all subodometers as factors; as well as the existence of a maximal subodometer factor, and a maximal odometer factor of a given minimal action.
\newline
\newline
\noindent \textit{Keywords.} 
Action,
Subodometer, 
Odometer, 
Rotation, 
Stone Space,
Eigenvalue, 
Ellis semigroup.  
\newline
\noindent \textit{2020 Mathematics Subject Classification.} 
Primary {37B05}; 
Secondary {54H15}, {20E18}. 
\end{abstract} \maketitle




\section{Introduction}
Minimal equicontinuous actions of a group $G$ on compact Hausdorff spaces form an essential part of the theory of topological dynamics. 
They were characterized by Auslander in \cite{auslander1988minimal} as actions of $G$ by translations on quotients $K/F$, where $K$ is a compactification of $G$ and $F$ is a closed subgroup of $K$. When $G$ is Abelian, every minimal equicontinuous system is a group rotation. However, when $G$ is non-Abelian, minimal equicontinuous systems that are not group rotations may arise. 

For $G=\mathbb{Z}$ a natural class of minimal rotations is given by the odometers, which are defined as inverse limits along sequences of finite minimal rotations \cite{downarowicz2005survey}. 
It is well-known that odometers $(X,\mathbb{Z})$ are (up to conjugacy) the minimal rotations of $\mathbb{Z}$ on metrizable Stone\footnote{
    A compact Hausdorff space is called \emph{Stone} if it is totally disconnected. 
    If $(X,G)$ is a minimal action, then $X$ is a metrizable Stone space if and only if it is finite or the Cantor space.} 
spaces \cite{hewitt1979abstract, downarowicz2005survey} and that similar statements hold for actions of finitely generated groups, such as $\mathbb{Z}^d$ \cite{cortez2006ZdToeplitz, cortez2008Godometers}. 

If $G$ is finitely generated, but not necessarily Abelian, a minimal equicontinuous action of $G$ on a Stone space is not necessarily a rotation. 
It follows from \cite{cortez2008Godometers, cortez2016orbit} that minimal equicontinuous actions on metrizable Stone spaces are given as inverse limits along sequences of finite minimal actions. 
Actions of this type are called \emph{subodometers} in \cite{cortez2008Godometers}. 
We will see in Section \ref{sec:metrizableSubodometers} that any minimal equicontinuous action of a residually finite group on a Stone space is metrizable. This observation allows us to extend the definitions of \cite{cortez2008Godometers} as follows. 

\begin{definition}
    Let $G$ be a group. 
    A minimal action $(X,G)$ on a Stone space is called
    \begin{itemize}
        \item[(i)] \emph{odometer} if $(X,G)$ is a group rotation. 
        \item[(ii)] \emph{subodometer} if $(X,G)$ is equicontinuous,
    \end{itemize}
\end{definition}

Note that unlike previous works, such as for example \cite{cortez2008Godometers, cortez2016orbit, giordano2019ZdOdometers, hurder2021limit, hurder2023prime}, we do not restrict ourselves to the case in which the phase space is metrizable, nor to the case in which the acting group is countable. 
Indeed, as we will see in Section \ref{sec:metrizableSubodometers}, the group $G$ admits a non-metrizable subodometer if and only if the collection of finite index subgroups of $G$ is uncountable. 
Examples of such groups are given by uncountable products of non-trivial finite groups, or a free group with infinitely many generators.  

In Sections \ref{sec:subodometersGeneratedByScales} and \ref{sec:metrizableSubodometers} we show that subodometers are inverse limits along nets of finite minimal actions, while metrizable subodometers are represented by the inverse limits along sequences of finite minimal actions. 
Analogous statements hold for odometers and finite minimal rotations as we will see in Section \ref{sec:odometers}.
Thus, odometers are actions on profinite groups \cite{ribes2000profinite}. 
It is important to observe that two odometers can have the same profinite group as a phase space, while being completely different as actions. For details, see Example \ref{exa:finfty} below.

In order to achieve these statements we develop a characterization of subodometers in terms of their uniformities in Section \ref{sec:entouragesEquivalenceRelationsAndGTiles}. 
To give a characterization of subodometers $(X,G)$ in terms of their topology we then explore the following notion. 
An open subset $B\subseteq X$ is called a $G$-tile if $\{g.B;\, g\in G\}$ is a partition of $X$. 
We then show that a minimal action $(X,G)$ is a subodometer if and only if the topology of $X$ allows for a base of $G$-tiles. 

It is natural to ask for a complete conjugacy invariant that allows to simplify the study of the category of subodometers and their respective factor maps. Such an invariant is given by the following.  
For an action $(X,G)$ we denote $\Eig(X,G)$ for the set of all finite index subgroups $\Gamma$ of $G$ for which there exists a factor map $\pi\colon X\to G/\Gamma$. 
In Section \ref{sec:eigenvalues} we show the following. 

\begin{theorem*}[Proposition \ref{pro:factorsOfSubodometers} and Theorem \ref{the:eigConjugacyInvariant}]
\leavevmode
    \begin{itemize}
        \item Any factor of a subodometer is also a subodometer.
        \item Let $(X,G)$ be a minimal action. 
        A subodometer $(Y,G)$ is a factor of $(X,G)$ if and only if $\Eig(Y,G)\subseteq \Eig(X,G).$
        \item Two subodometers are conjugated if and only if they have the same eigenvalues. 
    \end{itemize}     
\end{theorem*}

We will see that $\Eig$ naturally reflects properties of subodometers. 
In Section~\ref{sec:metrizableSubodometers} we show that a subodometer $(X,G)$ is metrizable if and only if $\Eig(X,G)$ is countable. 
Furthermore, in Section \ref{sec:odometers} we prove that a subodometer $(X,G)$ is an odometer if and only if $\Eig(X,G)$ is closed under finite intersections.
These observations will allow us to show the following in Section \ref{sec:eigensets}.

\begin{theorem}
\label{the:INTROOdometerLattice}
    Let $\mathfrak{X}$ be a non-empty family of odometers. 
    There exist minimal actions $\bigvee \mathfrak{X}$ and $\bigwedge \mathfrak{X}$ such that
    \begin{itemize}
        \item[(a$_1$)] $\bigvee\mathfrak{X}$ is an extension of all $(X,G)\in \mathfrak{X}$. 
        \item[(a$_2$)] $\bigwedge\mathfrak{X}$ is a factor of all $(X,G)\in \mathfrak{X}$.
        \item[(b$_1$)] Any minimal action $(Y,G)$ that is an extension of all $(X,G)\in \mathfrak{X}$ is an extension of $\bigvee \mathfrak{X}$. 
        \item[(b$_2$)] Any minimal action $(Y,G)$ that is a factor of all $(X,G)\in \mathfrak{X}$ is a factor of $\bigwedge \mathfrak{X}$. 
    \end{itemize}    
    With respect to the specified properties $\bigvee\mathfrak{X}$ and $\bigwedge \mathfrak{X}$ are unique up to conjugacy, and both are odometers. 
\end{theorem}

\begin{definition}
    Let $\mathfrak{X}$ be a non-empty family of odometers. 
    $\bigvee\mathfrak{X}$ is called the \emph{minimal common extension} (also \emph{supremum}) of $\mathfrak{X}$ and $\bigwedge \mathfrak{X}$ is called the \emph{maximal common factor} (also \emph{infimum}) of $\mathfrak{X}$. 
    For odometers $X,X'$ we abbreviate $X\vee X':=\bigvee \{X,X'\}$ and $X\wedge X':=\bigwedge \{X,X'\}$.     
\end{definition}

\begin{remark}
    We will see in Example \ref{exa:EigensetsNoLattice} that a similar statement as in Theorem \ref{the:INTROOdometerLattice} does not hold for non-empty families of subodometers.    
\end{remark}

It is well-known that normal subgroups satisfy the modular law, which we will discuss in Subsection \ref{subsec:eigensets_Modularity} in detail. We show that this modular law is inherited by the odometers, i.e.\ the following.  

\begin{theorem}
    \label{the:INTROmodularLattice}
    For odometers $X_1,X_2$ and $X$ for which $X_1$ is a factor of $X_2$ we have 
    \[(X_1 \vee X)\wedge X_2 = X_1 \vee (X\wedge X_2).\]
\end{theorem}

Clearly the maximal common factor of a (non-empty) family of metrizable odometers is metrizable. 
Concerning the minimal common extension we show the following in Subsection \ref{subsec:eigensets_SupremaOfMetrizableOdometers}. 

\begin{theorem}
    \label{the:INTROmetrizabilitySupremum}
    For a countable non-empty family $\mathfrak{X}$ of metrizable odometers also the minimal common extension $\bigvee \mathfrak{X}$ is metrizable. 
\end{theorem}

As an application of the developed theory we present in Section \ref{sec:universalSubodometers} (for a given group $G$) the existence of a \emph{universal odometer} that has all subodometers as factors. 
Furthermore, we establish that any minimal action has a \emph{maximal subodometer factor} and a \emph{maximal odometer factor} and discuss how respective inverse limit representations can be derived.
We prove that a minimal action is a subodometer if and only if its Ellis semigroup is an odometer. 
Furthermore, we show that the \emph{enveloping odometer}, defined as the infimum of all odometer extensions, is conjugated to the Ellis semigroup of a subodometer $(X,G)$. 
Using the latter, we present how an inverse limit representation of the Ellis semigroup of a subodometer $(X,G)$ can be derived from an inverse limit representation of $(X,G)$, and that a subodometer is metrizable if and only if its Ellis semigroup is metrizable. 

\subsection*{Convention:} Below $G$ denotes a (discrete) group if nothing else is mentioned.

\section{Preliminaries}
\label{sec:prelims}

Whenever $R$ is a relation on a set $X$ and $x\in X$ we denote \[R[x]:=\{x'\in X;\, (x',x)\in R\} 
~~\text{ and }~~ R^{-1}:=\{(x',x);\, (x,x')\in R\}.\] 
For relations $R$ and $R'$ on $X$ we denote $RR'$ for the set of all $(x,x')\in X^2$ for which there exists $y\in X$ with $(x,y)\in R$ and $(y,x')\in R'$. Note that this operation is associative. 

\subsection{Lattices}
A \emph{partially ordered set} is a pair $(P,\leq)$, where $P$ is a set and $\leq$ is a binary relation on $P$ that is reflexive, antisymmetric, and transitive.
Let $(P,\leq)$ be a partially ordered set. 
$(P,\leq)$ is called a \emph{lattice} if for every pair in $P$ the supremum\footnote{
    For $A\subseteq P$ an element $x\in P$ is an \emph{upper bound} of $A$ if $a\leq x$ for all $a\in A$.
    A \emph{supremum} of $A$ is an upper bound $x$ such that
    $x\leq y$ for every upper bound $y$ of $A$. Note that a supremum is unique, if it exists. 
    Similarly, we define lower bounds and infima. 
} and the infimum exist. 
$(P,\leq)$ is called a \emph{complete lattice} if any subset allows for a supremum and an infimum. 
A subset $L\subseteq P$ is called a \emph{complete sublattice} of $P$ if for any $A\subseteq L$, both the supremum and the infimum exist in $(P,\leq)$ and are contained in $L$.

\subsection{Topology}
\label{subsec:prelims_topology}
A subset of a topological space is called \emph{clopen} if it is closed and open. A relation on a topological space $X$ is called \emph{closed/open/clopen} if it is closed/open/clopen as a subset of $X^2$ equipped with the product topology. 

\subsubsection{Uniformities of compact Hausdorff spaces}
Let $X$ be a compact Hausdorff space. 
We denote $\mathbb{U}_X$ for the \emph{uniformity of $X$}, i.e.\ the set of all neighborhoods $\epsilon$ of the diagonal $\Delta_X:=\{(x,x);\, x\in X\}$ within $X^2$. 
The elements of $\mathbb{U}_X$ are called \emph{entourages}. 
An entourage $\epsilon\in \mathbb{U}_X$ is called symmetric if $\epsilon=\epsilon^{-1}$.
For all $\epsilon\in \mathbb{U}_X$ there exists a symmetric $\delta\in \mathbb{U}_X$ such that $\delta\delta\subseteq \epsilon$. 
For reference and details on uniformities see \cite[Chapter 6]{kelley2017general}. 

For $\epsilon\in \mathbb{U}_X$ we denote $B_\epsilon(x):=\epsilon[x]$. 
A subset $\mathbb{B}\subseteq \mathbb{U}_X$ is called a \emph{base for $\mathbb{U}_X$} whenever for each $\epsilon\in \mathbb{U}_X$ there exists $\delta\in \mathbb{B}$ with $\delta\subseteq \epsilon$. 
Whenever $\mathbb{B}$ is a base for $\mathbb{U}_X$, then for $x\in X$ the set $\{B_\epsilon(x);\, \epsilon\in \mathbb{B}\}$ is a (topological) neighborhood base of $x$ \cite[Corollary 30]{kelley2017general}. This allows one to recover the topology of $X$ from $\mathbb{U}_X$.

\subsubsection{Stone spaces}
A non-empty topological space $X$ is called \emph{totally disconnected} if the only connected subsets are the singletons \cite[Page 101]{illanes1999hyperspaces}.
It is called \emph{zero-dimensional}, whenever there exists a base of clopen sets for the topology \cite[Page 64]{illanes1999hyperspaces}. 
A compact Hausdorff space is totally disconnected if and only if it is zero-dimensional \cite[Theorem 12.11]{illanes1999hyperspaces}.
A \emph{Stone} space is a totally disconnected compact Hausdorff space. 
Note that non-empty closed subsets of Stone spaces are Stone spaces.

\subsection{Groups}
\label{subsec:prelims_groups}
Let $G$ be a group. 
For a subset $M\subseteq G$ we denote $\grouphull{M}$ for the subgroup generated by $M$. For subgroups $\Gamma$ and $\Lambda$ of $G$ we denote $\Gamma\leq \Lambda$ whenever $\Gamma$ is a subgroup of $\Lambda$. 
For a group homomorphism $\phi\colon G\to H$ we denote $\ker(\phi):=\{g\in G;\, \phi(g)=e_H\}$ for its \emph{kernel}.

A subgroup $\Gamma\leq G$ is said to be of finite index, whenever the partition
$G/\Gamma:=\{g\Gamma;\, g\in G\}$ is finite.
We denote $[G\colon \Gamma]:=|G/\Gamma|$ for its \emph{index (in $G$)}. 
We denote $\subfin(G)$ for the set of all finite index subgroups of $G$.
Note that $\subfin(G)$ is \emph{upward closed}, i.e.\ that for $\Gamma\in \subfin(G)$ and $\Lambda\leq G$ with $\Gamma\leq \Lambda$ we have $\Lambda\in \subfin(G)$. 

For $\Gamma\leq G$ a subset $F\subseteq G$ is called a \emph{fundamental domain}, whenever each $g\in G$ has a unique representation as a product $g=h\gamma$ with $(h,\gamma)\in F\times \Gamma$. Note that $F$ can be chosen to contain the identity element $e_G$. 
Whenever $\Gamma\in \subfin(G)$ we have $|F|=[G\colon \Gamma]<\infty$. 

\begin{remark}
\label{rem:fundamentalDomainsComposition}
    Whenever $F$ is a fundamental domain of $\Lambda\leq G$ and $F'$ is a fundamental domain of $\Gamma\leq \Lambda$, then $F'F$ is a fundamental domain of $\Gamma\leq G$. 
    Thus, for $\Gamma,\Lambda\in \subfin(G)$ with $\Gamma\leq \Lambda$ we have $[G\colon \Lambda][\Lambda \colon \Gamma]=[G\colon \Gamma]$. 
    In particular, whenever $\Gamma$ and $\Lambda$ have the same index in $G$ we have $[\Lambda\colon \Gamma]=1$, i.e.\ $\Gamma=\Lambda$. 
\end{remark}

For $\Gamma\leq G$ and $g\in G$ we denote\footnote{
Note that our definition of $\Gamma^g$ differs from that in \cite{ribes2000profinite}, where $\Gamma^g$ denotes $g^{-1}\Gamma g$. Our convention is chosen so that it interacts more naturally with left actions. This will become evident in the subsequent discussion.
} 
$\Gamma^g:=g\Gamma g^{-1}$ for the \emph{conjugate of $\Gamma$ by $g$}.
Note that $\Gamma^g\leq G$ and that $(\Gamma^g)^h=\Gamma^{(hg)}$ for $g,h\in G$. 
Two subgroups $\Gamma$ and $\Lambda$ are called \emph{conjugated} if there exists $g\in G$ with $\Gamma=\Lambda^g$. 
A subgroup $\Gamma\leq G$ is called \emph{normal (in $G$)} if $\Gamma^g=\Gamma$ for all $g\in G$. 
The \emph{normal core} of $\Gamma\leq G$ is given by $\Gamma_G:=\bigcap_{g\in G} \Gamma^g$. 
It is the largest subgroup of $\Gamma$ that is normal in $G$.
For further details we recommend \cite{ribes2000profinite}. 

\begin{remark}
\label{rem:fundamentalDomainsAndConjugation}
    Let $\Gamma\leq G$ and $F$ be a fundamental domain of $\Gamma$. 
    For $g\in G$ there exists $h\in F$ with $g\Gamma=h\Gamma$ and we observe 
    $\Gamma^g=(g\Gamma)(g\Gamma)^{-1}=(h\Gamma)(h\Gamma)^{-1}=\Gamma^{h}.$
    Thus $\{\Gamma^g;\, g\in G\}$ is finite and we observe that the normal core is a finite index subgroup. 
\end{remark}

\subsection{Actions}
\label{subsec:prelims_actions}
Let $G$ be a group and $X$ be a compact Hausdorff space. 
An \emph{action} of $G$ on $X$ is a group homomorphism $\alpha$ from $G$ into the group of homeomorphisms $X\to X$.
For $g\in G$ and $x\in X$, we suppress the symbol for the action by simply writing $g.x:=\alpha(g)(x)$. 
This allows us to simply speak of an \emph{action $(X,G)$}. 
See \cite{auslander1988minimal} for reference to the following and more details on actions.  

Let $(X,G)$ be an action. 
For $x\in X$ we denote $G.x:=\{g.x;\, g\in G\}$ for the \emph{orbit} of $x$. An action $(X,G)$ is called \emph{minimal} if all $x\in X$ have a dense orbit. 
For $x\in X$ we denote $G_0(x):=\{g\in G;\, g.x=x\}$ for the \emph{stabilizer of $x$}. 
Note that $G_0(g.x)=G_0(x)^g$ holds for all $g\in G$.

If $(X_i,G)_{i\in I}$ is a family of actions the \emph{product} $\prod_{i\in I} X_i$ is also a compact Hausdorff space. 
On $\prod_{i\in I} X_i$ we consider the action given by $g.(x_i)_{i\in I}:=(g.x_i)_{i\in I}$. 
For a subset $M\subseteq X$ and $g\in G$ we denote $g.M:=\{g.x;\, x\in M\}$. 
A subset $M\subseteq X$ is called \emph{invariant} if $g.M=M$ holds for all $g\in G$. 
A subset of $X^2$ is called \emph{invariant} if it is invariant w.r.t.\ $(X^2,G)$. 

\subsubsection{Equivariant maps}
A continuous mapping $\pi\colon X\to Y$ between actions $(X,G)$ and $(Y,G)$ is called \emph{equivariant} if $\pi(g.x)=g.\pi(x)$ holds for all $g\in G$ and $x\in X$. 
An equivariant homeomorphism is called a \emph{conjugacy}. 
Two actions are called \emph{conjugated} if there exists a conjugacy between them. 
An equivariant surjection is called a \emph{factor map}. 
$(X,G)$ is called an \emph{extension} of $(Y,G)$ and $(Y,G)$ is called a \emph{factor} of $(X,G)$ if there exists a factor map $\pi\colon X\to Y$. 
Any factor map is \emph{closed}, i.e.\ images of closed sets are closed. 
Any equivariant map between minimal actions is a factor map. 

If $R\subseteq X$ is an invariant closed equivalence relation, then $g.R[x]:=R[g.x]$ induces an action on $X\big/R$ and the quotient mapping $x\mapsto R[x]$ is a factor map.  
Furthermore, for a factor map $\pi\colon X\to Y$ we denote $R(\pi):=\{(x,x')\in X^2;\, \pi(x)=\pi(x')\}$ for the \emph{fibre relation}. 
$R(\pi)$ is an invariant closed equivalence relation and $(Y,G)$ is conjugated to $(X\big/R(\pi),G)$. 

\subsubsection{Equicontinuity}
An action $(X,G)$ is called \emph{equicontinuous} if, for any $\epsilon\in \mathbb{U}_X$ there exists $\delta\in \mathbb{U}_X$ such that for all $(x,x')\in \delta$ and $g\in G$ we have $(g.x,g.x')\in \epsilon$. 
Note that $(X,G)$ is equicontinuous if and only if $\mathbb{U}_X$ allows for a base consisting of invariant entourages. 
Minimal equicontinuous actions $(X,G)$ are \emph{coalescent}, i.e.\ any factor map $\pi\colon X\to X$ is a conjugacy \cite[Page 81]{auslander1988minimal}. In particular, minimal equicontinuous actions $(X,G)$ and $(Y,G)$ are conjugated, if and only if there exist factor maps $X\to Y$ and $Y\to X$. 
A minimal equicontinuous action $(X,G)$ is called \emph{regular}\footnote{
    Note that for minimal equicontinuous actions $(X,G)$ also $(X^2,G)$ is equicontinuous. 
    Thus for minimal equicontinuous actions the notion of regularity from \cite[Page 147]{auslander1988minimal} can be equivalently formulated as we did it.}, 
whenever for $x,x'\in X$ there exists a conjugacy $\iota\colon X\to X$ with $\iota(x)=x'$.

\subsubsection{Rotations}
Let $G$ be a group and $X$ be a compact (Hausdorff) group.  
A \emph{(group) rotation} is an action $(X,G)$ of the form
$g.x = \phi(g) x$,
where $\phi \colon G \to X$ is a group homomorphism. 
Any rotation is equicontinuous. 
A rotation is minimal if and only if $\phi(G)$ is dense in $X$.
Note that for a minimal rotation $\phi$ is uniquely determined by $\phi(g)=g.e_X$, where $e_X$ denotes the identity of $X$. 
Any minimal rotation $(X,G)$ is regular, since for $x,x'\in X$ we have that $y\mapsto yx^{-1}x'$ yields a conjugation $\iota\colon X\to X$ with $\iota(x)= x'$.

\subsubsection{Maximal equicontinuous factors}
\label{subsubsec:prelims_actions_MaximalEquicontinuousFactor}
For any action $(X,G)$ there exists a \emph{maximal equicontinuous factor}, i.e.\ a factor map $\pi_{\operatorname{MEF}}\colon X\to X_{\operatorname{MEF}}$ onto an equicontinuous action $(X_{\operatorname{MEF}},G)$ such that for any factor map $\pi\colon X\to Y$ onto an equicontinuous action $(Y,G)$ there exists a factor map $\psi\colon X_{\operatorname{MEF}}\to Y$ with $\pi=\psi \circ \pi_{\operatorname{MEF}}$. The maximal equicontinuous factor is unique up to conjugacy. For details see \cite{auslander1988minimal}. 

\subsubsection{Finite subodometers}
\label{subsubsec:prelims_actions_FiniteSubodometers}
An action $(X,G)$ is called \emph{finite} if $X$ is finite. 
Note that the finite minimal actions are exactly the finite subodometers. 
We next summarize some well-known statements about finite subodometers. 
For finite actions $(X,G)$ we have $G_0(x)\in \subfin(G)$. 

Let $(X,G)$ and $(Y,G)$ be finite subodometers and $(x,y)\in X\times Y$. It is straightforward to show that there exists a factor map $\pi\colon X\to Y$ with $\pi(x)=y$ if and only if $G_0(x)\subseteq G_0(y)$. 
Furthermore, there exists a conjugacy $\iota\colon X\to Y$ with $\iota(x)=y$ if and only if $G_0(x) = G_0(y)$. 

For $\Gamma\in \subfin(G)$ we denote $G/\Gamma$ for the finite partition $\{h\Gamma;\, h\in G\}$ and consider the induced action $(G/\Gamma,G)$ given by $g.(h\Gamma):=gh\Gamma$ for $g\in G$ and $h\Gamma\in G/\Gamma$.
Clearly, we have $G_0(\Gamma)=\Gamma$.
Thus, whenever $(X,G)$ is a finite subodometer and $x\in X$, then $(X,G)$ and $(G/G_0(x),G)$ are conjugated and the conjugacy can be chosen such that $x\mapsto G_0(x)$. 

Note that for $\Gamma,\Lambda\in \subfin(G)$ and $g\in G$ there exists a factor map\footnote{
    The factor map is given by
    $h\Lambda \mapsto hg\Gamma=(hgh^{-1})h\Gamma$.     
    } 
$\pi\colon G/\Lambda\to G/\Gamma$ with $\pi(\Lambda)=g\Gamma$ if and only if $\Lambda\subseteq \Gamma^g$.
Furthermore, there exists a conjugacy $\iota\colon G/\Lambda\to G/\Gamma$ with $\iota(\Lambda)=g\Gamma$ if and only if $\Lambda= \Gamma^g$.
In particular, $\Gamma$ and $\Lambda$ are conjugated if and only if $G/\Gamma$ and $G/\Lambda$ are conjugated. 
Recall that any rotation is regular. 
Thus, $\Gamma\in \subfin(G)$ is normal if and only if $(G/\Gamma,G)$ is a rotation, i.e.\ a finite odometer.

\section{Entourages, equivalence relations and $G$-tiles}
\label{sec:entouragesEquivalenceRelationsAndGTiles}

\subsection{Uniformities of equicontinuous actions on Stone spaces}
\label{subsec:entouragesEquivalenceRelationsAndGTiles_uniformitiesOfEquicontinuousActionsOnStoneSpaces}
In order to characterize subodometers by properties of their respective uniformities we are interested in entourages that are equivalence relations. The following lemma summarizes the relevant properties of such entourages.

\begin{lemma}
\label{lem:characterizationClosedEquivalenceRelationEntourage}
    Let $X$ be a compact Hausdorff space and $\rho\in \mathbb{U}_X$ an equivalence relation. 
    \begin{itemize}
        \item[(i)] $\rho$ establishes a finite partition. 
        \item[(ii)] $B_\rho(x)$ is clopen for all $x\in X$. 
        \item[(iii)] $\rho$ is clopen (in $X^2$). 
    \end{itemize}
\end{lemma}
\begin{proof}
    We first show that $B_\rho(x)$ is open for all $x\in X$. 
    Consider $x'\in B_\rho(x)$. 
    From $\rho\in \mathbb{U}_X$ we know that $B_\rho(x')$ is a neighborhood of $x'$. Furthermore, $\rho$ is an equivalence relation and hence $B_\rho(x')=B_\rho(x)$. 
    This shows $B_\rho(x)$ to be open. 

(i): 
    Since $\{B_\rho(x);\, x\in X\}$ is an open cover of $X$ and $X$ is compact there exists a finite $F\subseteq X$ with $X=\bigcup_{x\in F}B_\rho(x)=\bigcup_{x\in F}\rho[x]$. 
    In particular, the finite set $\{\rho[x];\, x\in F\}$ is the partition induced by $\rho$. 

(ii): 
    For $x\in X$ we already know that $B_\rho(x)$ is open. 
    Furthermore, it is the finite union of the other open equivalence classes and hence closed. 

(iii): 
    By (i) and (ii), $\rho=\bigcup_{x\in X} \rho[x]^2$ is clopen as a finite union of clopen subsets of $X^2$.
\end{proof}

The following characterizes subodometers via properties of the respective uniformities.  

\begin{proposition}
\label{pro:characterizationEquicontinuousStone}
    Let $(X,G)$ be an action. 
    \begin{itemize}
        \item[(i)] $X$ is a Stone space if and only if $\mathbb{U}_X$ allows for a base consisting of equivalence relations. 
        \item[(ii)] $(X,G)$ is an equicontinuous action on a Stone space if and only if $\mathbb{U}_X$ allows for a base consisting of invariant equivalence relations. 
    \end{itemize}
\end{proposition}
\begin{proof}
(i): 
    Assume first that $\mathbb{U}_X$ allows for a base of equivalence relations. 
    Consider $x\in X$ and a neighborhood $U$ of $x$. 
    There exists $\epsilon\in \mathbb{U}_X$ such that $B_\epsilon(x)\subseteq U$. 
    By our assumption, we find an equivalence relation $\rho\in \mathbb{U}_X$ with $\rho\subseteq \epsilon$. 
    Clearly, we have $B_\rho(x)\subseteq B_\epsilon(x)\subseteq U$. 
    Furthermore, it follows from Lemma \ref{lem:characterizationClosedEquivalenceRelationEntourage} that $B_\rho(x)$ is clopen. 
    This shows that $X$ is zero-dimensional, i.e.\ a Stone space. 

    For the converse assume that $X$ is a Stone space and consider $\epsilon\in \mathbb{U}_X$. 
    Let $\delta\in \mathbb{U}_X$ be symmetric and such that $\delta\delta\subseteq \epsilon$. 
    Since $X$ is a Stone space for $x\in X$ there exists $A_x$ clopen with $x\in A_x\subseteq B_\delta(x)$. 
    We have $X=\bigcup_{x\in X}A_x$ and the compactness of $X$ yields that $X=\bigcup_{x\in F}A_x$ for some finite subset $F\subseteq X$. 
    Enumerating $F=\{x_1,x_2,\dots, x_n\}$ and denoting $A_1:=A_{x_1}$ and 
    $A_k:=A_{x_k}\setminus \bigcup_{i=1}^{k-1} A_i$ for $k\in \{2,\dots ,n\}$ we obtain a partition $\{A_1,\dots, A_n\}$ of $X$ that consists of clopen sets. 
    Thus $\rho:=\bigcup_{k=1}^n A_k^2$ is a clopen equivalence relation.
    In particular, we have $\rho\in \mathbb{U}_X$. 
    For $(y,y')\in \rho$ there exists $k\in \{1,\dots, n\}$ such that $(y,y')\in A_k^2\subseteq A_{x_k}^2\subseteq B_\delta(x_k)^2\subseteq \delta\delta\subseteq \epsilon$.
    This shows $\rho \subseteq \epsilon$. 

(ii): 
    Recall that an action is equicontinuous if and only if $\mathbb{U}_X$ allows for a base consisting of invariant entourages. 
    Thus, from (i) we observe that whenever $\mathbb{U}_X$ allows for a base consisting of invariant equivalence relations, then $(X,G)$ is equicontinuous and $X$ is a Stone space. 

    For the converse assume that $(X,G)$ is an equicontinuous action on a Stone space and consider $\epsilon\in \mathbb{U}_X$. 
    By (i) there exists an equivalence relation $\rho'\in \mathbb{U}_X$ with $\rho'\subseteq \epsilon$. Denote $\rho:=\bigcap_{g\in G}g.\rho'$ and note that $\rho$ is an invariant equivalence relation with $\rho\subseteq\rho'\subseteq \epsilon$.
    Since $(X,G)$ is equicontinuous there exists an invariant $\delta\in \mathbb{U}_X$ with $\delta\subseteq\rho'$ and it follows from 
    $\delta=\bigcap_{g\in G}g.\delta\subseteq \rho$ that $\rho\in \mathbb{U}_X$. 
\end{proof}

\subsection{$G$-tiles}
\label{subsec:entouragesEquivalenceRelationsAndGTiles_GTiles}
The following notion will be an essential tool in the study of subodometers. 

\begin{definition}
    Let $(X,G)$ be an action. 
    An open subset $A\subseteq X$ is called a \emph{$G$-tile} if $\{g.A;\, g\in G\}$ is a partition of $X$. We refer to this partition as the \emph{partition induced by $A$} and to the respective equivalence relation as the \emph{equivalence relation induced by $A$}. 
\end{definition}

\begin{remark}
    Since $X$ is compact, the induced partition of a $G$-tile is finite and hence any $G$-tile is clopen.     
\end{remark}

For minimal actions $G$-tiles and invariant equivalence relations that are entourages are closely related. 

\begin{lemma}
\label{lem:GtilesVSInvariantEquivalenceRelationEntourages}
    Let $(X,G)$ be a minimal action. 
    \begin{itemize}
        \item[(i)] The equivalence relation $\rho$ induced by a $G$-tile $A$ is an invariant clopen entourage and satisfies $B_\rho(x)=A$ for all $x\in A$.
        \item[(ii)] If $\rho\in \mathbb{U}_X$ is an invariant equivalence relation, then $B_\rho(x)$ is a $G$-tile for any $x\in X$. 
    \end{itemize}
\end{lemma}
\begin{proof}
(i): 
    From $A\in \{g.A;\, g\in G\}$ we observe that $B_\rho(x)=\rho[x]=A$ holds for all $x\in A$. 
    Choose $F\subseteq X$ finite with $\{g.A;\, g\in F\}=\{g.A;\, g\in G\}$ and note that $\rho=\bigcup_{g\in F}g.A^2$ is the finite union of clopen sets and hence clopen. 
    Since $\rho$ is reflexive we observe $\rho$ to be an entourage. 
    For $g\in G$ and $(x,x')\in \rho$ there exists $h\in G$ with $(x,x')\in h.A^2$. 
    It follows that $(g.x,g.x')\in (gh).A^2\subseteq \rho$, which shows $\rho$ to be invariant. 
    
(ii):
    Let $x\in X$. 
    It follows from Lemma \ref{lem:characterizationClosedEquivalenceRelationEntourage} that $B_\rho(x')$ is open for all $x'\in X$. 
    Thus, the minimality of $(X,G)$ yields that for $x'\in X$ there exists $g\in G$ with $g.x\in B_\rho(x')$. Since $\rho$ is an invariant equivalence relation we observe $x'\in B_\rho(g.x)=g.B_\rho(x)$. 
    This shows $\bigcup_{g\in G} g.B_\rho(x)=X$ and establishes $\{g.B_\rho(x);\, g\in G\}=\{\rho[g.x];\, g\in G\}$ as the partition given by the equivalence relation $\rho$. 
\end{proof}

\subsection{Characterization of subodometers via $G$-tiles}
\label{subsec:entouragesEquivalenceRelationsAndGTiles_CharacterizationOfSubodometersViaGTiles}

\begin{theorem}
\label{the:characterizationGTiles}
    A minimal action $(X,G)$ is a subodometer if and only if the topology of $X$ allows for a base of $G$-tiles. 
\end{theorem}
\begin{proof}
    If $(X,G)$ is a subodometer, then we know from Lemma \ref{lem:characterizationClosedEquivalenceRelationEntourage} and Proposition \ref{pro:characterizationEquicontinuousStone} that $\mathbb{U}_X$ allows for a base $\mathbb{B}$ of invariant closed equivalence relations. 
    From Lemma \ref{lem:GtilesVSInvariantEquivalenceRelationEntourages} we observe that $\{B_\epsilon(x);\, (x,\epsilon)\in X\times \mathbb{B}\}$ is a base for the topology of $X$ that consists of $G$-tiles. 

    For the converse assume that $X$ allows for a base of $G$-tiles for its topology. By Proposition \ref{pro:characterizationEquicontinuousStone} it suffices to show that $\mathbb{U}_X$ allows for a base of invariant equivalence relations. 
    Let $\epsilon\in \mathbb{U}_X$. 
    Consider a symmetric entourage $\delta\in \mathbb{U}_X$ with $\delta\delta\subseteq \epsilon$. 
    For $x\in X$ choose a $G$-tile with $x\in A_x\subseteq B_\delta(x)$ and recall that $G$-tiles are open.  
    Since $X$ is compact we observe that there exists a finite subset $F\subseteq X$ such that $\bigcup_{x\in F}A_x=X$. 
    For $x\in F$ denote $\rho_x$ for the equivalence relation induced by $A_x$. From Lemma \ref{lem:GtilesVSInvariantEquivalenceRelationEntourages} we know that $\rho_x\in \mathbb{U}_X$ is a clopen invariant equivalence relation.
    Since $F$ is finite also $\rho:=\bigcap_{x\in F} \rho_x$ is a clopen invariant equivalence relation. In particular, we have $\rho \in \mathbb{U}_X$. 
    To show that $\rho\subseteq \epsilon$ consider $(x_1,x_2)\in \rho$.  
    There exists $x\in F$ such that $x_1\in A_x=\rho_x[x]$. 
    From $(x_1,x_2)\in \rho\subseteq \rho_x$ we observe that 
    $x_2\in \rho_x[x_1]=\rho_x[x]=A_x$ and hence 
    $(x_1,x_2)\in A_x^2\subseteq B_\delta(x)^2\subseteq \delta \delta \subseteq \epsilon$.
\end{proof}

\section{Subodometers generated by scales}
\label{sec:subodometersGeneratedByScales}
We next show that any subodometer is conjugated to an inverse limit of finite minimal actions. For this we will need the following notions. 
Recall that $\subfin(G)$ denotes the set of all finite index subgroups $\Gamma\leq G$. 

\begin{definition}
    A subset $S\subseteq \subfin(G)$ is called 
    \begin{itemize}
        \item \emph{a scale} if $S$ is non-empty and for $\Gamma,\Gamma'\in S$ there exists $\Lambda\in S$ with $\Lambda\subseteq \Gamma\cap \Gamma'$. 
        \item \emph{upward closed} if for all $\Gamma\in S$ and $\Lambda\in \subfin(G)$ with $\Gamma\subseteq \Lambda$ we have $\Lambda\in S$. 
        \item \emph{a filter} if it is an upward closed scale. 
    \end{itemize}
\end{definition}

\begin{remark}
    Our definition of scale differs from the definition of scale considered in \cite{cortez2008Godometers}, where scales are given by sequences $(\Gamma_n)_{n\in \mathbb{N}}$ in $\subfin(G)$ with $\Gamma_{n+1}\subseteq \Gamma_n$. 
    We will see in Section \ref{sec:metrizableSubodometers} below that these sequences can be used to characterize metrizable subodometers. 
\end{remark}

\begin{remark}
    A subset $S\subseteq \subfin(G)$ is called \emph{intersection closed} (also \emph{intersection stable}) if for all $\Gamma,\Gamma'\in S$ we have $\Gamma\cap \Gamma'\in S$. 
    Any intersection closed subset of $\subfin(G)$ is a scale. Furthermore, any filter is intersection closed. 
\end{remark}

Let $S\subseteq \subfin(G)$ be a scale.
Note that for $\Gamma_1,\Gamma_2\in S$ with $\Gamma_1\leq \Gamma_2$ there exists a unique factor map $\pi_{\Gamma_2}^{\Gamma_1}\colon G/\Gamma_1 \to G/\Gamma_2$ with $\Gamma_1\mapsto \Gamma_2$. 
For $\Gamma_1,\Gamma_2,\Gamma_3\in \subfin(G)$ with $\Gamma_1\leq \Gamma_2\leq \Gamma_3$ we have $\pi_{\Gamma_3}^{\Gamma_1}=\pi_{\Gamma_3}^{\Gamma_2}\circ \pi_{\Gamma_2}^{\Gamma_1}$. 
Thus, any scale $S$ induces an inverse system $(\pi_{\Lambda}^{\Gamma})_{(\Gamma,\Lambda)\in S_*^2}$ of factor maps, where we abbreviate 
$S_*^2:=\{(\Gamma,\Lambda)\in S^2;\, \Gamma\leq \Lambda\}$.
See \cite[Section 1.1]{ribes2000profinite} for further details on inverse systems and inverse limits. 

\begin{definition}
    Let $S\subseteq \subfin(G)$ be a scale. 
    The \emph{$S$-subdometer} is given by the inverse limit 
    \[
    \varprojlim_{\Gamma \in S} G\big/\Gamma 
    := 
    \Bigl\{ (x_\Gamma)_{\Gamma \in S} \in \prod_{\Gamma \in S} G\big/\Gamma 
    ;\, 
    \pi_{\Lambda}^\Gamma(x_\Gamma) = x_{\Lambda} \text{ for all } (\Gamma,\Lambda)\in S_*^2
    \Bigr\}.
    \] 
    For a scale $S$ we say that a subodometer $(X,G)$ is \emph{generated by $S$}, whenever it is conjugated to the $S$-subodometer.     
\end{definition}

\begin{remark}
    Note that the $S$-subodometer is the minimal component of the equicontinuous action of $G$ on $\prod_{\Gamma \in S} G/\Gamma$ that contains $(\Gamma)_{\Gamma\in S}$. 
    Since $\prod_{\Gamma \in S} G/\Gamma$ is a Stone space we observe that $(\varprojlim_{\Gamma \in S} G/\Gamma,G)$ is a subodometer, justifying the terminology.     
\end{remark}

We will next establish that any subodometer is generated by some scale. 
For this we will need the following. 

\begin{definition}
    Let $(X,G)$ be a minimal action. 
    For $x\in X$ we denote $\Eig_x(X,G)$ for the set of all $\Gamma\in \subfin(G)$ for which there exists a factor map $\pi\colon X\to G/\Gamma$ with $\pi(x)=\Gamma$. 
\end{definition}

\begin{theorem}
\label{the:scalesForSubodometersViaEigx}
    Let $(X,G)$ be a minimal action and $x\in X$. 
    \begin{itemize}
        \item[(i)] $\Eig_x(X,G)$ is a filter. 
        \item[(ii)] If $(X,G)$ is a subodometer then
        $(X,G)$ is generated by $\Eig_x(X,G)$. 
    \end{itemize}
\end{theorem}

For the proof of Theorem \ref{the:scalesForSubodometersViaEigx} we first study hitting times. 
Let $(X,G)$ be a minimal action. 
For $x\in X$ and $B\subseteq X$ we denote 
\[G(x,B):=\{g\in G;\, g.x\in B\}\] for the set of \emph{hitting times} of $x$ to $B$. 

\begin{lemma}
\label{lem:EigxDescriptionGTiles}
    Let $(X,G)$ be a minimal action and $x\in X$. 
    For any $x\in X$ we have that 
    $\Eig_x(X,G)=\{G(x,B);\, B~G\text{-tile with } x\in B\}$.
\end{lemma}
\begin{proof}
'$\subseteq$':
    Consider $\Gamma\in \Eig_x(X,G)$ and a factor map $\pi\colon X\to G/\Gamma$ with $\pi(x)=\Gamma$. 
    Denote $\rho:=R(\pi)$ for the fibre relation and note that $\rho$ is an invariant closed equivalence relation. 
    Since $G/\Gamma$ is finite it establishes finitely many equivalence classes which are easily observed to be clopen. 
    It follows that $\rho$ is a finite union of clopen sets and hence clopen. 
    In particular $\rho$ is an entourage and Lemma \ref{lem:GtilesVSInvariantEquivalenceRelationEntourages} yields that $B:=B_\rho(x)$ is a $G$-tile containing $x$ that satisfies 
    $G(x,B)=G(\pi(x),\pi(B))=G(\Gamma,\{\Gamma\})=G_0(\Gamma)=\Gamma.$

'$\supseteq$':
    Consider a $G$-tile $B$ with $x\in B$. 
    Denote $\rho$ for the induced equivalence relation and recall from Lemma \ref{lem:GtilesVSInvariantEquivalenceRelationEntourages} that $\rho$ is an invariant clopen entourage. 
    In particular, it is an invariant closed equivalence relation. 
    Consider the factor map $\pi\colon X\to G/\rho$ and note that $G/\rho$ is finite. 
    We thus have 
    \begin{align*}
        \Gamma:=G(x,B)=G(\pi(x),\pi(B))=G(\pi(x),\{\pi(x)\})=G_0(\pi(x))\in \subfin(G).
    \end{align*}
    It follows from $G_0(\pi(x))=\Gamma$ that there exists a conjugacy $\iota\colon X\big/\rho\to G/\Gamma$ with $\iota(\pi(x))=\Gamma$. 
    This shows $G(x,B)=\Gamma\in \Eig_x(X,G)$. 
\end{proof}

To show that $\Eig_x(X,G)$ is a scale it is convenient to first establish the following. 

\begin{lemma}
\label{lem:GTileNeighbourhoodsIntersectionClosed}
    Let $(X,G)$ be a minimal action and $x\in X$. 
    For $G$-tiles $B$ and $B'$ containing $x$ there exists a $G$-tile $B''$ with $x\in B''\subseteq B\cap B'$. 
\end{lemma}
\begin{proof}
    Denote $\rho$ and $\rho'$ for the equivalence relations induced by $B$ and $B'$, respectively. 
    Denote $\xi:=\rho\cap \rho'$ and note that $\xi$ is an equivalence relation. 
    It follows from Lemma \ref{lem:GtilesVSInvariantEquivalenceRelationEntourages} that $\rho$ and $\rho'$ are invariant clopen entourages and hence also $\xi$ is an invariant clopen entourage. 
    Thus, by Lemma \ref{lem:GtilesVSInvariantEquivalenceRelationEntourages} 
     $B'':=B_\xi(x)\subseteq B\cap B'$ is a $G$-tile that contains $x$. 
\end{proof}

We will also need the following. 

\begin{lemma}
\label{lem:existenceFactorMaps}
    Let $(X,G)$ and $(Y,G)$ be equicontinuous minimal actions and consider $x\in X$ and $y\in Y$. 
    There exists a factor map $\pi\colon X\to Y$ with $\pi(x)=y$ if and only if for any neighborhood $V$ of $y$ there exists a neighborhood $U$ of $x$ such that 
    $G(x,U)\subseteq G(y,V)$. 
\end{lemma}
\begin{proof}
    Since $\pi$ is continuous the first property implies the second. 
    For the converse note that the second property implies that $G_0(x)\subseteq G_0(y)$. 
    This observation allows us to consider the mapping 
    $\phi\colon G.x\to Y$ given by $g.x\mapsto g.y$.
    By minimality $G.x$ is dense in $X$ and $\phi$ has a dense image. 
    We next show that $\phi$ is \emph{uniformly continuous on $G.x$}, i.e.\ that for any $\epsilon\in \mathbb{U}_Y$ there exists $\delta\in \mathbb{U}_X$ such that
    for all $x_1,x_2\in G.x$ with $(x_1,x_2)\in \delta$ we have $(\phi(x_1),\phi(x_2))\in \epsilon$. 
    This will allow to extend $\phi$ to a continuous map $\pi\colon X\to Y$ \cite[Theorem 6.26]{kelley2017general}. It is straightforward to verify that $\pi$ is then a factor map that satisfies $\pi(x)=y$. 

    To show the uniform continuity of $\phi$ on $G.x$ 
    consider $\epsilon\in \mathbb{U}_Y$. 
    Since $X$ is equicontinuous, we assume w.l.o.g.\ that $\epsilon$ is invariant. 
    We observe that $V:=B_\epsilon(y)$ is a neighborhood of $y$ and by our assumption there exists a neighborhood $U$ of $x$, such that $G(x,U)\subseteq G(y,V)$. 
    Choose an invariant $\delta\in \mathbb{U}_X$ such that $B_\delta(x)\subseteq U$. 
    For $x_1,x_2 \in G.x$ with $(x_1,x_2)\in \delta$ there exist $g_1,g_2\in G$ with $g_i.x=x_i$. 
    Since $\delta$ is invariant we have 
    $g_2^{-1}g_1.x\in B_\delta(x)\subseteq U$ 
    and hence 
    $g_2^{-1}g_1\in G(x,U)\subseteq G(y,V)$. 
    It follows that
    $g_2^{-1}g_1.y\in V=B_\epsilon(y)$ 
    and the invariance of $\epsilon$ yields
    $(\phi(x_1),\phi(x_2))=(g_1.y,g_2.y)\in \epsilon.$
\end{proof}

\begin{proof}[Proof of Theorem \ref{the:scalesForSubodometersViaEigx}:]
(i):
    Let $(X,G)$ be a minimal action and $x\in X$. 
    Clearly, $\Eig_x(X,G)$ is upward closed. 
    Since $B\mapsto G(x,B)$ is monotone w.r.t.\ set inclusion we observe $\Eig_x(X,G)$ to be a scale from the Lemmas \ref{lem:EigxDescriptionGTiles} and \ref{lem:GTileNeighbourhoodsIntersectionClosed}. 

(ii):
    Let $(X,G)$ be a subodometer and $x\in X$. 
    Abbreviate $S:=\Eig_x(X,G)$ and denote $(X',G)$ for the $S$-odometer and $x':=(\Gamma)_{\Gamma\in S}$. 
    Since minimal equicontinuous actions are coalescent it suffices to show that $(X,G)$ and $(X',G)$ are mutual factors of each other. 

    For each $\Gamma\in S$ denote $\pi_\Gamma\colon X\to G/\Gamma$ for the factor map with $\pi_\Gamma(x)=\Gamma$. 
    Consider the map $\hat{\pi}\colon X\to \prod_{\Gamma\in S} G/\Gamma$ given by $y\mapsto (\pi_\Gamma(y))_{\Gamma\in S}$ and note that $\hat{\pi}$ is an equivariant embedding that satisfies $\hat{\pi}(x)=x'$. 
    In particular, the image of $\hat{\pi}$ is given by $X'$ and the restriction of $\hat{\pi}$ establishes a factor map $\pi\colon X\to X'$.

    To show that $(X,G)$ is a factor of $(X',G)$ we use Lemma \ref{lem:existenceFactorMaps}. Let $V$ be a neighborhood of $x$. 
    By Theorem \ref{the:characterizationGTiles} there exists a $G$-tile $B$ with $x\in B\subseteq V$. 
    From Lemma \ref{lem:EigxDescriptionGTiles} we know that 
    $\Lambda:=G(x,B)\in \Eig_x(X,G)=S$. 
    Denote $\phi_\Lambda\colon X'\to G/\Lambda$ for the restriction of the projection $\prod_{\Gamma\in S}G/\Gamma\to G/\Lambda$ and note that $\phi_\Lambda$ is a factor map with $\phi_\Lambda(x')=\Lambda$. 
    Thus, $U:=\phi_\Lambda^{-1}(\Lambda)$ is an open neighborhood of $x'$ that satisfies 
    \begin{align*}
        G(x',U)
        &=G(\phi_\Lambda(x'),\phi_\Lambda(U))
        =G(\Lambda,\{\Lambda\})
        =G_0(\Lambda)\\
        &=\Lambda
        =G(x,B)
        \subseteq G(x,V).
    \end{align*}
    This shows that $(X,G)$ is a factor of $(X',G)$
\end{proof}

\section{Factors and eigenvalues}
\label{sec:eigenvalues}

\subsection{Factors of subodometers}
\begin{proposition}
\label{pro:factorsOfSubodometers}
    Any factor of a subodometer is also a subodometer. 
\end{proposition}
\begin{proof}
    Let $(X,G)$ be a subodometer and consider a factor map $\pi\colon X\to Y$. 
    Since $(X,G)$ is minimal and equicontinuous, so is $(Y,G)$. 
    To show that $Y$ is a Stone space, let $\mathcal{B}$ be a clopen base for the topology of $X$. 
    As a factor map $\pi$ is closed. 
    Furthermore, factor maps between minimal equicontinuous actions are open \cite[Theorem 7.3]{auslander1988minimal}. 
    Thus, $\mathcal{B}':=\{\pi(B);\, B\in \mathcal{B}\}$ consists of clopen sets. 
    To show that $\mathcal{B}'$ is a base for the topology of $Y$ consider $y\in Y$ and an open neighborhood $U$ of $y$. 
    Choose $x\in \pi^{-1}(y)$. 
    Since $\pi^{-1}(U)$ is open we find $B\in \mathcal{B}$ with $x\in B\subseteq \pi^{-1}(U)$. 
    We thus have $y=\pi(x)\in \pi(B)\subseteq \pi(\pi^{-1}(U))=U$. 
    This establishes $\mathcal{B}'$ as a base for the topology of $Y$. 
\end{proof}

\subsection{Eigenvalues of minimal actions}
\label{subsec:eigenvalues_EigenvaluesOfMinimalActions}

\begin{definition}
    Let $(X,G)$ be a minimal action. 
    A finite index subgroup $\Gamma\leq G$ is called an \emph{eigenvalue} of $(X,G)$, whenever there exists a factor map $\pi\colon X\to G/\Gamma$. We denote $\Eig(X,G)$ for the set of all eigenvalues of $(X,G)$. 
\end{definition}

\begin{remark}
    In the case of $\mathbb{Z}$-actions, an \emph{eigenvalue} is usually defined as a real number $\alpha$ such that there exists a factor map from the system $(X,\mathbb{Z})$ onto the rotation of the circle by angle $\alpha$. When $\alpha\in\mathbb{Q}$, one says that $\alpha$ is a \emph{rational eigenvalue}. In this setting, what we call an eigenvalue in this article corresponds to subgroups of the form $p\mathbb{Z}\subseteq \mathbb{Z}$ for which $1/p$ is a rational eigenvalue. See, for instance, \cite{ormes1997strong} for results involving the rational eigenvalues of a minimal $\mathbb{Z}$-action on the Cantor set.
\end{remark}

Note that whenever $(Y,G)$ is a factor of $(X,G)$, then $\Eig(Y,G)$ is contained in $\Eig(X,G)$. In particular conjugated minimal actions have the same eigenvalues.
In general $\Eig$ is not a complete conjugacy invariant as illustrated by the next example. 

\begin{example}
    Let $(X,G)$ be an action for which $X$ is connected. 
    Any factor of $X$ is connected and hence $\Eig(X,G)=\{G\}$. 
    Thus $X$ and the trivial action on one point have the same eigenvalues. 
\end{example}

We next show that $\Eig$ is a complete conjugacy invariant for subodometers.  

\begin{theorem}
\label{the:eigConjugacyInvariant}
    \begin{itemize}
        \item[(i)] Let $(X,G)$ be a minimal action. 
        A subodometer $(Y,G)$ is a factor of $(X,G)$ if and only if $\Eig(Y,G)\subseteq \Eig(X,G).$
        \item[(ii)] Two subodometers are conjugated if and only if they have the same eigenvalues. 
    \end{itemize}     
\end{theorem}

For the proof we first observe the following. 

\begin{lemma}
\label{lem:finitelyManyFactorMaps}
    Let $\pi\colon X\to Y$ be a factor map. If $Y$ is finite, then there are at most $|Y|$ factor maps $X\to Y$. 
\end{lemma}
\begin{proof}
    Fix $x_0\in X$. 
    Since $X$ is minimal any factor map $\pi\colon X\to Y$ is determined by $\pi(x_0)$.
    For the latter we have at most $|Y|$ many choices. 
\end{proof}

\begin{proof}[Proof of Theorem \ref{the:eigConjugacyInvariant}:]
Note that (ii) follows from (i), since minimal equicontinuous actions are coalescent. In order to show (i) it remains to show that $\Eig(Y,G)\subseteq \Eig(X,G)$ implies that $(Y,G)$ is a factor of $(X,G)$. 
    By Theorem \ref{the:scalesForSubodometersViaEigx} we assume w.l.o.g.\ that $(Y,G)$ is the $S$-subodometer for some scale $S\subseteq \subfin(G)$. 

    For $\Gamma\in S$ we denote $A_\Gamma$ for the set of all $x\in X$ for which there exists a factor map $\pi\colon X\to G/\Gamma$ with $\pi(x)=\Gamma$. 
    Since $\Gamma\in \Eig(X,G)$ we observe $A_\Gamma$ to be non-empty. 
    Recall from Lemma \ref{lem:finitelyManyFactorMaps} that there exist only finitely many factor maps $X\to G/\Gamma$. 
    Thus, the union
    $A_\Gamma = \bigcup_\pi \pi^{-1}(\Gamma),$
    taken over all such factor maps $\pi$, is a finite union of closed subsets of $X$ and hence closed. 
    This shows that $\mathcal{A}:=\{A_\Gamma;\, \Gamma\in S\}$ consists of non-empty closed subsets of $X$. 

    Consider $\Gamma,\Lambda\in S$ with $\Gamma\subseteq \Lambda$. 
    The canonical factor map $\pi_{\Lambda}^\Gamma\colon G/\Gamma \to G/\Lambda$ satisfies $\pi_{\Lambda}^\Gamma(\Gamma) = \Lambda$ and we observe $A_\Gamma\subseteq A_{\Lambda}$. 
    This shows that the mapping $S\ni \Gamma \mapsto A_\Gamma\in \mathcal{A}$ is monotone.
    Since $S$ is a scale it is straightforward to observe that $\mathcal{A}$ has the \emph{finite intersection property}, i.e.\ that any finite subset $\mathcal{F}\subseteq \mathcal{A}$ satisfies $\bigcap_{A\in \mathcal{F}}A\neq \emptyset$. 
    Since $X$ is compact, we observe that $A:=\bigcap_{A\in \mathcal{A}}A$ is non-empty \cite[Theorem 26.9]{munkres2000topology}. 
    
    Fix $x_0\in A$. 
    For $\Gamma\in S$ we have $x_0\in A_\Gamma$ and hence there exists a factor map $\pi_\Gamma\colon X\to G/\Gamma$ with $\pi_\Gamma(x_0)=\Gamma$. 
    We observe that 
    $\hat{\pi}\colon X\to \prod_{\Gamma\in S} G/\Gamma$ defined by $x\mapsto (\pi_\Gamma(x))$ has $\hat{\pi}(x_0)=(\Gamma)_{\Gamma\in S}$ in its image and hence restricts to a factor map $\pi\colon X\to Y$.   
\end{proof}

\subsection{From scales to eigenvalues}
\label{subsec:eigenvalues_FromScalesToEigenvalues}
Given a scale $S$, it is natural to ask how the eigenvalues of the $S$-subodometer can be derived from $S$. 
To answer this question in detail we introduce the following notions.

\begin{definition}
    Let $S\subseteq \subfin(G)$ be a subset. 
    \begin{itemize}
    \item 
        The \emph{upper hull} $\upperhull{S}$ of $S$ is defined as the set of all $\Gamma\in \subfin(G)$ for which there exists $\Lambda\in S$ with $\Lambda\subseteq \Gamma$. 
    \item 
        The \emph{conjugation hull} $\conhull{S}$ of $S$ is defined as the set of all $\Gamma\in \subfin(G)$ for which there exists $g\in G$ with $\Gamma^g\in S$. 
    \item 
        The \emph{upper conjugation hull} $\eigenhull{S}$ of $S$ is defined as the set of all $\Gamma\in \subfin(G)$ for which there exist $(g,\Lambda)\in G\times S$ with $\Lambda\subseteq \Gamma^g$, i.e.\ by $\eigenhull{S}:=\conhull{(\upperhull{S})}$. 
    \end{itemize}
\end{definition}

\begin{remark}
    \label{rem:upperHullScaleIsFilter}
    For a scale $S$ the upper hull $\upperhull{S}$ is a filter. 
\end{remark}

\begin{remark}
    It is straightforward to show that the intersection $\bigcap_{S\in \mathcal{S}} S$ and the union $\bigcup_{S\in \mathcal{S}} S$ of a family $\mathcal{S}$ of upward closed subsets of $\subfin(G)$ are also upward closed. 
    For a non-empty subset $S\subseteq \subfin(G)$ the upper hull $\upperhull{S}$ is the smallest upward closed subset of $\subfin(G)$ that contains $S$. 
\end{remark}

\begin{definition}
    A subset $S\subseteq \subfin(G)$ is called \emph{conjugation invariant} whenever for $\Gamma\in S$ and $g\in G$ we have $\Gamma^g\in S$. 
\end{definition}

\begin{remark}
    It is straightforward to show that the intersection $\bigcap_{S\in \mathcal{S}} S$ and the union $\bigcup_{S\in \mathcal{S}} S$ of a family $\mathcal{S}$ of conjugation invariant subsets of $\subfin(G)$ are also conjugation invariant. 
    For a non-empty subset $S\subseteq \subfin(G)$ the conjugation hull $\conhull{S}$ is the smallest conjugation invariant subset of $\subfin(G)$ that contains $S$. 
\end{remark}

\begin{remark}
\label{rem:upperConjugationHullProperties}
    Let $S\subseteq \subfin(G)$. 
    Whenever $S$ is conjugation invariant, then also $\upperhull{S}$ is conjugation invariant. 
    Whenever $S$ is upward closed, then also $\conhull{S}$ is upward closed. 
    This allows to observe that $\eigenhull{S}=\upperhull{(\conhull{S})}=\conhull{(\upperhull{S})}$ is the smallest upward closed and conjugation invariant subset of $\subfin(G)$ that contains $S$. 
\end{remark}

\begin{proposition}
\label{pro:conhullOfLocalEigenvalues}
    Let $(X,G)$ be a minimal action. 
    $\Eig(X,G)$ is upward closed and conjugation invariant. 
    For $x\in X$ we have 
    \[\Eig(X,G)=\conhull{\Eig_x(X,G)}.\] 
\end{proposition}
\begin{proof}
    It is straightforward to observe that $\Eig(X,G)$ is upward closed and conjugation invariant. 
    Thus, from $\Eig_x(X,G)\subseteq \Eig(X,G)$ we have $\conhull{\Eig_x(X,G)}\subseteq \Eig(X,G)$. 

    For the converse consider $\Gamma\in \Eig(X,G)$ and a factor map $\pi\colon X\to G/\Gamma$. 
    Let $g\in G$ be such that $\pi(x)=g\Gamma$. 
    As discussed in Subsection \ref{subsec:prelims_actions} there exists a conjugacy $\iota\colon G/\Gamma\to G/\Gamma^g$ with $\iota(g\Gamma)=\Gamma^g$. 
    Thus, the factor map $\phi:=\iota\circ \pi\colon X\to G/\Gamma^g$ satisfies $\phi(x)=\Gamma^g$. 
    This shows $\Gamma^g\in \Eig_x(X,G)$ and hence $\Gamma\in \conhull{\Eig_x(X,G)}$. 
\end{proof}

The transition from a scale to the eigenvalues of the respective subodometer can be described as follows. 

\begin{proposition}
\label{pro:scaleEigensetTransition}
    Let $S\subseteq \subfin(G)$ be a scale and denote $(X,G)$ for the $S$-odometer. Denote $x:=(\Gamma)_{\Gamma\in S}$ and note that $x\in X$.
    We have 
    \begin{itemize}
        \item[(i)] $\Eig_x(X,G)=\upperhull{S}$ and
        \item[(ii)] $\Eig(X,G) = \eigenhull{S}$. 
    \end{itemize}
\end{proposition}
\begin{proof}
(i):
    Recall that $X$ is given as the minimal component of $\hat{X}:=\prod_{\Gamma\in S}G/\Gamma$ containing $x$. 
    Denote $\hat{\pi}_\Gamma\colon \hat{X}\to G/\Gamma$ for the respective projections. 
    The restriction $\pi_\Gamma\colon X\to G/\Gamma$ is a factor map with $\pi_\Gamma(x)=\Gamma$. 
    We thus have $\Gamma\in \Eig_x(X,G)$ and observe $S\subseteq \Eig_x(X,G)$. 
    From Theorem \ref{the:scalesForSubodometersViaEigx} we know that $\Eig_x(X,G)$ is a filter. Since filters are upward closed we observe that $\upperhull{S}\subseteq \Eig_x(X,G)$. 
    
    For the converse consider $\Gamma\in \Eig_x(X,G)$ and a factor map $\pi\colon X\to G/\Gamma$ with $\pi(x)=\Gamma$. 
    From \cite[Lemma 1.1.16]{ribes2000profinite} it follows that there exist $\Lambda\in S$ and a continuous surjection $\phi\colon G/\Lambda\to G/\Gamma$ with $\pi= \phi\circ \pi_\Lambda$. 
    To see that $\phi$ inherits the equivariance from $\pi$ and $\pi_\Lambda$ consider $g\in G$ and $y\in G/\Lambda$. 
    For $x'\in \pi_\Lambda^{-1}(y)$ we have
    \begin{align*}
        \phi(g.y)
        =\phi(g.\pi_\Lambda(x'))
        =\phi(\pi_\Lambda(g.x'))
        =\pi(g.x')
        =g.\pi(x')
        =g.\phi(y).
    \end{align*}
    This establishes $\phi$ as a factor map. 
    Furthermore, we have 
    \begin{align*}
        \phi(\Lambda)
        =\phi(\pi_\Lambda(x))
        =\pi(x)
        =\Gamma
    \end{align*}
    and hence 
    $
        \Lambda
        =G_0(\Lambda)
        \subseteq G_0(\phi(\Lambda))
        =G_0(\Gamma)
        =\Gamma.
    $
    This shows $\Gamma\in \upperhull{S}$. 
    
(ii): 
    From (i) and Proposition \ref{pro:conhullOfLocalEigenvalues} we observe 
    \begin{align*}
        \Eig(X,G)
        &=\conhull{\Eig_x(X,G)}
        =\conhull{(\upperhull{S})}
        =\eigenhull{S}.
        \qedhere
    \end{align*}
\end{proof}

If $S$ is a scale and $(X,G)$ a subodometer with $\Eig(X,G)=\eigenhull{S}$, then the $S$-subodometer $(X',G)$ satisfies $\Eig(X',G)=\eigenhull{S}=\Eig(X,G)$ and Theorem \ref{the:eigConjugacyInvariant} yields that $(X,G)$ and $(X',G)$ are conjugated. This shows the following. 

\begin{corollary}
\label{cor:generationAndEigenhull}
    A subodometer $(X,G)$ is generated by a scale $S$ if and only if $\Eig(X,G)=\eigenhull{S}$. 
\end{corollary}

\subsection{Equivalent scales}
\label{subsec:eigenvalues_EquivalentScales}
    Combining Theorem \ref{the:eigConjugacyInvariant} and Proposition \ref{pro:scaleEigensetTransition} also yields the following. 

    \begin{corollary}
    \label{cor:equivalentScales}
        Let $S$ and $S'$ be scales that generate $(X,G)$ and $(X',G)$, respectively. 
        \begin{itemize}
            \item[(i)] $(X,G)$ is a factor of $(X',G)$ if and only if $\eigenhull{S}\subseteq \eigenhull{S'}$. 
            \item[(ii)] $(X,G)$ and $(X',G)$ are conjugated if and only if $\eigenhull{S} = \eigenhull{S'}$. 
        \end{itemize}
    \end{corollary}

    To gain a sufficient condition for $\eigenhull{S}=\eigenhull{S'}$ we introduce the following. 

    \begin{definition}
        Consider subsets $S,S'\subseteq \subfin(G)$. 
        We say that $S'$ \emph{dominates $S$} if for $\Gamma\in S$ there exists $\Gamma'\in S'$ with $\Gamma'\subseteq \Gamma$. 
        We say that $S$ and $S'$ are \emph{equivalent} if $S$ dominates $S'$ and $S'$ dominates $S$. 
    \end{definition}

    \begin{remark}
    \label{rem:equivalentScalesYieldSameSubodometer}
        $S'$ dominates $S$ if and only if $\upperhull{S}\subseteq \upperhull{(S')}$.   
        In particular, if a scale $S'$ dominates a scale $S$, then $\eigenhull{S}\subseteq \eigenhull{S'}$ and hence the $S$-subodometer is a factor of the $S'$-subodometer. 
        
        It also follows that $S$ and $S'$ are equivalent if and only if $\upperhull{S}=\upperhull{(S')}$. Thus for equivalent scales $S$ and $S'$ we have $\eigenhull{S}=\eigenhull{S'}$ and hence the respective subodometers are conjugated. 
    \end{remark}

    Note that equivalence of scales is not necessary for inducing conjugated subodometers as illustrated by the next example. 

    \begin{example}
        Let $G$ be a group and $\Gamma\leq G$ be a non-normal finite index subgroup. Let $g\in G$ with $\Gamma^g\neq \Gamma$. 
        Note that $S:=\{\Gamma\}$ and $S':=\{\Gamma^g\}$ are scales that are not equivalent.
        Nevertheless, we have $\conhull{S}=\conhull{(S')}$ and hence $\eigenhull{S}=\upperhull{(\conhull{S})}=\upperhull{(\conhull{(S')})}=\eigenhull{S}$. Thus, the respective subodometers are conjugated. 
    \end{example}

\section{Metrizable subodometers}
\label{sec:metrizableSubodometers}

\subsection{Finite subodometers}
\label{subsec:metrizableSubodometers_FiniteSubodometers}

In order to characterize metrizable subodometers we first study finite subodometers. 

\begin{lemma}
\label{lem:characterizationFiniteSubodometers}
    A subodometer $(X,G)$ is finite if and only if $\Eig(X,G)$ is finite. 
\end{lemma}
\begin{proof}
    We assume w.l.o.g.\ that $(X,G)$ is the $S$-subodometer for some scale $S$. 
    Clearly, whenever $\Eig(X,G)$ is finite, then also $S$ is finite and hence $X$ is finite as a subset of the finite space $\prod_{\Gamma\in S}G/\Gamma$. 
    
    For the converse assume that $X$ is finite and choose $x\in X$.  
    It follows that $X$ only contains finitely many $G$-tiles and Lemma \ref{lem:EigxDescriptionGTiles} implies that $\Eig_x(X,G)$ is finite. 
    Since any finite index subgroup only has finitely many conjugates it follows from Proposition \ref{pro:conhullOfLocalEigenvalues} that $\Eig(X,G)=\conhull{\Eig_x(X,G)}$ is finite.     
\end{proof}

From Proposition \ref{pro:scaleEigensetTransition} and Lemma \ref{lem:characterizationFiniteSubodometers} we observe the following. 

\begin{corollary}
\label{cor:eigenhullFinite}
    For any $\Gamma\in \subfin(G)$ we have that $\eigenhull{\{\Gamma\}}$ is finite. 
\end{corollary}

\subsection{Chains}
\label{subsec:metrizableSubodometers_Chains}

In the literature on subodometers, such as \cite{cortez2008Godometers} only scales of the form $(\Gamma_n)_{n\in \mathbb{N}}$ in $\subfin(G)$ with $\Gamma_{n+1}\subseteq \Gamma_n$ are considered. 
We will show in Theorem \ref{the:characterizationMetrizableSubodometers} below that the inverse limits along this particular type of scale yield exactly the metrizable subodometers. 
We begin our discussion with the following definition. 

\begin{definition}
    A non-empty subset $S\subseteq \subfin(G)$ is called a \emph{chain} if for any $\Gamma,\Lambda\in S$ we have $\Gamma\subseteq \Lambda$ or $\Lambda\subseteq \Gamma$. 
\end{definition}

Clearly any chain is a scale. Chains are exactly the type of scales considered in the literature as we present next. 

\begin{lemma}
\label{lem:chainsAndSequences}
    Any chain $S$ is countable and can be written in the form $S=\{\Gamma_n;\, n\in \mathbb{N}\}$ with $\Gamma_{n+1}\subseteq \Gamma_n$ for all $n\in \mathbb{N}$.  
\end{lemma}
\begin{proof}
    Let $S$ be a chain and consider $i\colon S\to \mathbb{N}\cup \{0\}$ given by $\Gamma\mapsto [G\colon \Gamma]$. 
    Consider $\Gamma,\Lambda\in S$ with $i(\Gamma)=i(\Lambda)$. 
    Since $S$ is a chain we have $\Gamma\subseteq \Lambda$ or $\Lambda\subseteq \Gamma$. 
    It follows from Remark \ref{rem:fundamentalDomainsComposition} that $\Gamma=\Lambda$. This shows $i$ to be injective and hence $S$ inherits the countability from $\mathbb{N}\cup\{0\}$. 

    Clearly, whenever $S$ is finite, then $S$ can be written in the form $S=\{\Gamma_n;\, n\in \mathbb{N}\}$ with $\Gamma_{n+1}\subseteq \Gamma_n$ for all $n\in \mathbb{N}$. 
    If $S$ is infinite let $(k_n)_{n\in \mathbb{N}}$ be a strictly increasing sequence in $\mathbb{N}$ with $i(S)=\{k_n;\, n\in \mathbb{N}\}$. 
    For $n\in \mathbb{N}$ denote $\Gamma_n$ for the unique preimage of $k_n$ under $i$. 
    Clearly, we have $S=\{\Gamma_n;\, n\in \mathbb{N}\}$. 
    Since $i$ is monotone we observe $\Gamma_{n+1}\subseteq\Gamma_n$.
\end{proof}

The following insight will allow us to show that countable scales and chains give the same type of subodometer. 

\begin{lemma}
\label{lem:countableScaleHasEquivalentChain}
    For any countable scale there exists an equivalent chain.
\end{lemma}
\begin{proof}
    Let $S$ be a countable scale and enumerate $S=\{\Gamma_n;\, n\in \mathbb{N}\}$. 
    For $n\in \mathbb{N}$ define $\Lambda_n:=\bigcap_{k=1}^n\Gamma_k$ and note that $S':=\{\Lambda_n;\, n\in \mathbb{N}\}$ is a chain. 
    From $\Lambda_n\subseteq \Gamma_n$ we observe that $S'$ dominates $S$. 
    Furthermore, since $S$ is a scale, for $n\in \mathbb{N}$ we find $m\in \mathbb{N}$ such that $\Gamma_m\subseteq \bigcap_{k=1}^n\Gamma_k=\Lambda_n$. This shows that $S$ dominates $S'$. 
\end{proof}

\subsection{Characterization of metrizable subodometers}
\label{subsec:metrizableSubodometers_CharacterizationOfMetrizableSubodometers}

\begin{theorem}
\label{the:characterizationMetrizableSubodometers}
    Let $(X,G)$ be a subodometer. 
    The following statements are equivalent. 
    \begin{itemize}
        \item[(i)] $X$ is metrizable. 
        \item[(ii)] $(X,G)$ is generated by a chain. 
        \item[(iii)] $(X,G)$ is generated by a countable scale. 
        \item[(iv)] $\Eig(X,G)$ is countable.
        \item[(v)] The topology of $X$ has a countable base of $G$-tiles. 
        \item[(vi)] $\mathbb{U}_X$ has a countable base consisting of invariant equivalence relations. 
    \end{itemize}
\end{theorem}

For the proof of Theorem \ref{the:characterizationMetrizableSubodometers} we need the following Lemma. 

\begin{lemma}
\label{lem:eigenhullInheritsCountability}
    For any countable subset $S\subseteq \subfin(G)$ also $\eigenhull{S}$ is countable. 
\end{lemma}
\begin{proof}
    For $\Gamma\in S$ we know from Corollary \ref{cor:eigenhullFinite} that $\eigenhull{\{\Gamma\}}$ is finite. 
    Furthermore, it follows from the definition of the upper conjugation hull that $\eigenhull{S}=\bigcup_{\Gamma\in S}\eigenhull{\{\Gamma\}}$. 
    We thus observe $\eigenhull{S}$ to be a countable union of finite sets and hence to be countable. 
\end{proof}
\begin{proof}[Proof of Theorem \ref{the:characterizationMetrizableSubodometers}:]
Note that a compact Hausdorff space is metrizable if and only if its uniformity allows for a countable base \cite[Theorem 6.13]{kelley2017general}. Thus, the equivalence of (i) and (vi) can be observed from Proposition \ref{pro:characterizationEquicontinuousStone}. 

(i)$\Rightarrow$(v):
    If $X$ is metrizable it allows for a countable base $\mathcal{B}$ for its topology. It follows from Theorem \ref{the:characterizationGTiles} that $\mathcal{B}$ can be chosen to consist of $G$-tiles.  

(v)$\Rightarrow$(iii):
    Let $x\in X$ and $\mathcal{B}_x$ be a countable neighborhood base that consists of $G$-tiles. 
    Denote $S:=\{G(x,B);\, B\in \mathcal{B}_x\}$. 
    Clearly, $S$ is countable. 
    From Lemma \ref{lem:EigxDescriptionGTiles} we know $S\subseteq \Eig_x(X,G)\subseteq \subfin(G)$.        
    Furthermore, for $B,B'\in \mathcal{B}_x$ we have that $B\cap B'$ is an open neighborhood of $x$ and hence there exists $B''\in \mathcal{B}_x$ with $B''\subseteq B\cap B'$. 
    It follows that $G(x,B'')\subseteq G(x,B)\cap G(x,B')$ and we observe that $S$ is a countable scale.
    It remains to show that $S$ generates $(X,G)$. 
    For this we show that $S$ and $\Eig_x(X,G)$ are equivalent.

    We already know that $S\subseteq \Eig_x(X,G)$ and it remains to show that $S$ dominates $\Eig_x(X,G)$. 
    From Lemma \ref{lem:EigxDescriptionGTiles} it follows that for any $\Gamma\in \Eig_x(X,G)$ there exists a $G$-tile $A$ with $x\in A$ and $\Gamma= G(x,A)$. 
    Since $\mathcal{B}_x$ is a neighborhood base and $G$-tiles are open we find $B\in \mathcal{B}_x$ with $B\subseteq A$. 
    From $S\ni G(x,B)\subseteq G(x,A)=\Gamma$ we observe that $S$ dominates $\Eig_x(X,G)$.    

(iii)$\Rightarrow$(i): 
    Let $S$ be a countable scale that generates $(X,G)$. 
    As a subspace of a countable product of metrizable spaces, we observe the $S$-subodometer to be metrizable. 
    Thus also the conjugated subodometer $(X,G)$ is metrizable. 

This establishes the equivalence of (i), (iii), (v) and (vi). 

(iii)$\Rightarrow$(ii):
    If $S$ is a countable scale that generates $(X,G)$, then by Lemma \ref{lem:countableScaleHasEquivalentChain} there exists an equivalent chain $S'$. It follows from Remark \ref{rem:equivalentScalesYieldSameSubodometer} that $S'$ also generates $(X,G)$. 

(ii)$\Rightarrow$(iv):
    If $S$ is a countable chain that generates $(X,G)$, then Corollary \ref{cor:generationAndEigenhull} and Lemma \ref{lem:eigenhullInheritsCountability} yield that 
    $\Eig(X,G)=\eigenhull{S}$ is countable. 

(iv)$\Rightarrow$(iii): 
    If $\Eig(X,G)$ is countable, then any scale that generates $(X,G)$ is contained in $\Eig(X,G)$ and hence countable. 
\end{proof}

\begin{corollary}
\label{cor:allSubodometersMetrizable}
    Let $G$ be a group. 
    All subodometers $(X,G)$ are metrizable if and only if $\subfin(G)$ is countable. 
\end{corollary}
\begin{proof}
    Note that $S:=\subfin(G)$ is a filter. 
    The $S$-subodometer $(X,G)$ satisfies $\Eig(X,G)=\eigenhull{S}=\subfin(G)$. 
    If $(X,G)$ is metrizable, then Theorem \ref{the:characterizationMetrizableSubodometers} yields that $\subfin(G)$ is countable. 

    For the converse assume that $\subfin(G)$ is countable. 
    For any subodometer $(X,G)$ we have that $\Eig(X,G)\subseteq \subfin(G)$ is countable and Theorem \ref{the:characterizationMetrizableSubodometers} yields that $(X,G)$ is metrizable. 
\end{proof}

It is natural to ask for which groups $\subfin(G)$ is countable. We next present some examples and counterexamples.

\begin{example}
    For a finitely generated group $G$ any finite index subgroup is also finitely generated \cite[Corollary 7.2.1]{hall1959theory} and hence $\subfin(G)$ is countable. 
    In particular, $\subfin(\mathbb{Z}^d)$ is countable for $d\in \mathbb{N}$. 
\end{example}

\begin{example}
    The Abelian group $(\mathbb{Q},+)$ has only itself as a finite index subgroup, i.e.\ $\subfin(\mathbb{Q})=\{\mathbb{Q}\}$. 
    Thus, all (sub)odometers $(X,\mathbb{Q})$ are trivial. 
\end{example}

\begin{example}
    For an uncountable product $G=\prod_i G_i$ of finite non-trivial groups $\subfin(G)$ is uncountable, since it contains all $\Gamma_i:=\prod_j G_{j,i}$ with $G_{i,i}=\{e_{G_i}\}$ and $G_{j,i}:=G_j$ for $i\neq j$.
\end{example} 

\begin{example}
    The free group $F_\infty$ with countably many generators is countable, while $\subfin(F_\infty)$ is uncountable. For details see Example \ref{exa:finfty} below. 
\end{example}

\section{Odometers}
\label{sec:odometers}
Note that a totally disconnected compact Hausdorff group is profinite \cite{ribes2000profinite}. 
Thus, odometers are group rotations on profinite groups. 
The reader might wonder whether the theory of odometers is not just a reminiscent of the theory of profinite groups. 
The following examples illustrate that this is not the case. 

\begin{example}
\label{exa:finfty}
    Let $F_\infty$ be the free group generated by countably infinitely many generators $(g_n)_{n\in \mathbb{N}}$ and note that $F_\infty$ is a countable group. 
    Consider the finite group $X:=\mathbb{Z}/2\mathbb{Z}$ and denote $I:=X^\mathbb{N}\setminus \{(0)_{n\in \mathbb{N}}\}$. 
    For a sequence $h=(h_n)_{n\in \mathbb{N}}\in I$ we consider the group homomorphism $\phi_h\colon F_\infty\to X$ given by $\phi_h(g_n)=h_n$ for all $n\in \mathbb{N}$ and denote $\alpha_h$ for the respective group rotation.
    Note that $\phi_h$ is surjective and hence that $\alpha_h$ is a finite odometer. 
    Furthermore, w.r.t.\ this action $G_0(x)=\ker(\phi_h)$ holds for all $x\in X$. 
    Thus, in order to show that for distinct $h,h'\in I$ the odometers $\alpha_h$ and $\alpha_{h'}$ are not conjugated it suffices to show that $\ker(\phi_h)\neq\ker(\phi_{h'})$. 

    For $h,h'\in I$ with $h\neq h'$ there exists $n\in \mathbb{N}$ with $h_n\neq h_n'$ and we observe that $g_n$ is contained in exactly one of the normal sets $\ker(\phi_h)$ and $\ker(\phi_{h'})$. 
    Thus, indeed $\ker(\phi_h)\neq \ker(\phi_{h'})$ and we observe that $\alpha_h$ and $\alpha_{h'}$ are not conjugated.
    Since $I$ is uncountable this shows that there exist uncountably many non-conjugated odometers $(X,F_\infty)$ with $|X|=2$. 
    In contrast, up to group-isomorphy there exists only one group with two elements. 
\end{example}

\begin{example}\label{ex:amenable_non_amenable} There exist finitely generated groups $G$ and $H$, with $G$ amenable and $H$ non-amenable, whose profinite completions are isomorphic \cite{kionke2023amenability}. Consequently, if $X_1$ and $X_2$ denote the profinite completions of $G$ and $H$, respectively, then $X_1$ and $X_2$ are isomorphic as topological groups \cite[Theorem~1.1]{nikolov2007finitely}. Nevertheless, the odometers $(X_1,G)$ and $(X_2,H)$ are not isomorphic as dynamical systems, nor are they even continuously orbit equivalent (see \cite{cortez2016orbit}).
 \end{example}

We next aim to characterize the odometers by properties of their sets of eigenvalues and properties of the scales that generate them. 
For this we introduce the following notions. 

\subsection{Normality and core-stability}
\label{subsec:odometers_NormalityAndCoreStability}

\begin{definition}
    A subset $S\subseteq \subfin(G)$ is called
    \begin{itemize}
        \item \emph{normal} if all $\Gamma\in S$ are normal. 
        \item \emph{core-stable} if for $\Gamma\in S$ also the core $\Gamma_G$ is contained in $S$. 
    \end{itemize}
\end{definition}

It is natural to ask how these notions relate to conjugation invariance. 
We summarize the relations in the following proposition and omit the straightforward proofs. 

\begin{proposition}
\label{pro:coreStableNormalConIinvInterplay}
    \begin{itemize}
        \item[(i)] Any normal scale is core-stable and conjugation invariant. 
        \item[(ii)] A chain is normal if and only if it is conjugation invariant. Any conjugation invariant chain is core-stable. 
        \item[(iii)] A filter is core-stable if and only if it is conjugation invariant. 
    \end{itemize}
\end{proposition}

\begin{example}
    Let $G$ be a group that allows for non-normal finite index subgroups $\Lambda \leq \Gamma\leq G$ with $\Lambda\neq \Gamma$.
    \begin{itemize}
        \item[(i)]
        The chain $S:=\{\Gamma, \Gamma_G\}$ is core-stable but not normal.
        \item[(ii)]
        The filter $S:=\upperhull{\{\Gamma_G\}}$ is conjugation invariant, i.e.\ core-stable. It contains $\Gamma$ and hence is not normal. 
        \item[(iii)] 
        The scale $S:=\conhull{\{\Gamma,\Lambda, \Lambda_G\}}$ is conjugation invariant but not core-stable.  
    \end{itemize}    
\end{example}

\subsection{Characterization of odometers}
\label{subsec:odometers_CharacterizationOfOdometers}

\begin{theorem}
\label{the:characterizationOdometers}
    Let $(X,G)$ be a subodometer. 
    The following statements are equivalent. 
    \begin{itemize}
        \item[(i)] $(X,G)$ is an odometer.         
        \item[(ii)] $(X,G)$ is generated by a normal scale. 
        \item[(iii)] $(X,G)$ is generated by a conjugation invariant scale. 
        \item[(iv)] $(X,G)$ is generated by a core-stable scale. 
        \item[(v)] $\Eig(X,G)$ is a filter. 
        \item[(vi)] $\Eig(X,G)$ is core-stable. 
        \item[(vii)] For some $x\in X$ we have $\Eig(X,G)=\Eig_x(X,G)$. 
        \item[(viii)] For all $x\in X$ we have $\Eig(X,G)=\Eig_x(X,G)$. 
    \end{itemize}
\end{theorem}
\begin{proof}
(i)$\Rightarrow$(viii):
    Since odometers are group rotations they are regular. 
    Thus, $\Eig_x(X,G)$ is independent of $x\in X$ and we observe 
    \[\Eig(X,G)=\bigcup_{x'\in X}\Eig_{x'}(X,G)=\Eig_x(X,G)\] for all $x\in X$. 

(viii)$\Rightarrow$(vii): 
    Trivial. 

(vii)$\Rightarrow$(v): 
    Assume that there exists $x\in X$ with $\Eig(X,G)=\Eig_x(X,G)$. 
    It follows from Theorem \ref{the:scalesForSubodometersViaEigx} that $\Eig(X,G)$ is a filter. 

(v)$\Rightarrow$(vi): 
    Assume that $\Eig(X,G)$ is a filter and recall from Proposition \ref{pro:conhullOfLocalEigenvalues} that $\Eig(X,G)$ is conjugation invariant. It follows from Proposition \ref{pro:coreStableNormalConIinvInterplay} that $\Eig(X,G)$ is core-stable. 

(vi)$\Rightarrow$(ii): 
    Let $S$ be a scale that generates $(X,G)$. 
    Consider \[S':=\{\Gamma_G;\, \Gamma\in S\}\] and note that $S'$ consists of normal subgroups of $G$. 
    It follows from the monotonicity of $\Gamma\mapsto \Gamma_G$ that $S'$ is a normal scale. 
    Clearly, $S'$ dominates $S$ and hence $\Eig(X,G)=\eigenhull{S}\subseteq \eigenhull{S'}$. 
    Furthermore, the core-stability of $\Eig(X,G)$ implies $S'\subseteq \Eig(X,G)$. 
    Since $S'$ is conjugation invariant and $\Eig(X,G)$ is upward closed we have $\eigenhull{S'}=\upperhull{(S')}\subseteq \Eig(X,G)$.
    This shows $\Eig(X,G)=\eigenhull{S'}$ and hence $S'$ generates $(X,G)$.  

(ii)$\Rightarrow$(i):
    If $S$ is a normal scale that generates $(X,G)$, then the $S$-subodometer is an inverse limit of group rotations and hence a group rotation. We thus observe $(X,G)$ to be an odometer.

This shows that (i), (ii), (v), (vi), (vii) and (viii) are equivalent.

(ii)$\Rightarrow$(iv):
    Trivial.

(iv)$\Rightarrow$(iii):  
    Let $S$ be a core-stable scale that generates $(X,G)$. 
    Since $S$ is a scale $\upperhull{S}$ is a filter. 
    For $\Gamma\in \upperhull{S}$ and $g\in G$ we find $\Lambda\in S$ with $\Lambda\subseteq \Gamma$ and since $S$ is core-stable we have 
    $S\ni\Lambda_G\subseteq \Gamma^g$. 
    It follows that $\Gamma^g\in \upperhull{S}$. 
    This shows $\upperhull{S}$ to be conjugation invariant. 
    Since $S$ generates $(X,G)$ we have $\Eig(X,G)=\eigenhull{S}=\conhull{(\upperhull{S})}=\eigenhull{\upperhull{S}}$. 
    Thus, the conjugation invariant scale $\upperhull{S}$ generates $(X,G)$.     

(iii)$\Rightarrow$(v):
    Let $S$ be a conjugation invariant scale that generates $(X,G)$. 
    Note that $S=\conhull{S}$. 
    Thus, from Proposition \ref{pro:scaleEigensetTransition} we observe that 
    $\Eig(X,G)=\eigenhull{S}=\upperhull{(\conhull{S})}=\upperhull{S}$. 
    Since $S$ is a scale $\upperhull{S}=\Eig(X,G)$ is a filter. 
\end{proof}

\begin{corollary}
    For a subodometer $(X,G)$ the following statements are equivalent. 
    \begin{itemize}
        \item[(i)] $(X,G)$ is a metrizable odometer. 
        \item[(ii)] $(X,G)$ is generated by a normal chain. 
        \item[(iii)] $\Eig(X,G)$ is a countable filter. 
    \end{itemize}
\end{corollary}
\begin{proof}
    From combining the Theorems \ref{the:characterizationMetrizableSubodometers} and \ref{the:characterizationOdometers} it follows that (ii) implies (i) and that (i) implies (iii). 

    To show that (iii) implies (ii) assume that $\Eig(X,G)$ is a countable filter. 
    It follows from Theorem \ref{the:characterizationOdometers} that 
    $(X,G)$ allows for a normal scale $S$.
    From $S\subseteq \Eig(X,G)$ we observe $S$ to be countable.     
    Note that the finite intersection of normal subgroups is normal. 
    Thus, the construction presented in the proof of Lemma \ref{lem:countableScaleHasEquivalentChain} yields an equivalent normal chain $S'$ that also generates $(X,G)$. 
\end{proof}

\begin{remark}
    Clearly, whenever $G$ is Abelian, then any subgroup of $G$ is normal and hence any subodometer is an odometer. 
    Note that non-Abelian groups for which all subgroups are normal do exist and are called \emph{Hamiltonian groups}. A characterization of Hamiltonian groups can be found in \cite[Theorem 12.5.6]{hall1959theory}. 
\end{remark}

\section{Eigensets}
\label{sec:eigensets}

We have seen in Theorem \ref{the:eigConjugacyInvariant} that subodometers $(X,G)$ are (up to conjugacy) completely characterized by $\Eig(X,G)$.
Furthermore, in Proposition \ref{pro:scaleEigensetTransition} we have seen that for a scale $S$ and the $S$-odometer $(X,G)$ we have $\Eig(X,G)=\eigenhull{S}$. 
This motivates the following purely algebraic definitions, which allow for a simplification of the study of the category of subodometers and in particular simple proofs for the Theorems \ref{the:INTROOdometerLattice}, \ref{the:INTROmodularLattice}, and 
\ref{the:INTROmetrizabilitySupremum}. 

\begin{definition}
    A subset $E\subseteq \subfin(G)$ is called an \emph{eigenset} if there exists a scale $S\subseteq \subfin(G)$ such that $E=\eigenhull{S}$. 
    A scale $S$ is said to \emph{generate $E$} if $E=\eigenhull{S}$. 
    An eigenset is called \emph{filtered} if it is a filter. 
    We denote $\allEigensets{G}$ and $\allFilteredEigensets{G}$ for the set of all eigensets and all filtered eigensets, respectively.
    We equip $\allEigensets{G}$ and $\allFilteredEigensets{G}$ with the partial ordering given by set inclusion. 
\end{definition}

\begin{remark}
\label{rem:eigensetsViaMinimalActions}
    For any minimal action $(X,G)$, the set $\Eig(X,G)$ is an eigenset. 
    Indeed, it follows from Theorem \ref{the:scalesForSubodometersViaEigx} that $\Eig_x(X,G)$ is a filter, and Proposition \ref{pro:conhullOfLocalEigenvalues} yields
    $\Eig(X,G)=\conhull{\Eig_x(X,G)}=\eigenhull{\Eig_x(X,G)}$. 
\end{remark}

It follows from the discussion in Section \ref{sec:eigenvalues} that for each eigenset $E$ there exists a unique subodometer (up to conjugacy) with $E=\Eig(X,G)$. We refer to this subodometer as the \emph{subodometer associated with $E$}.     
For $E,E'\in\allEigensets{G}$ we have $E\subseteq E'$ if and only if the subodometer associated with $E$ is a factor of the subodometer associated with $E'$ (Theorem \ref{the:eigConjugacyInvariant}). 
Furthermore, an eigenset is generated by a scale $S$ if and only if the associated subodometer is generated by $S$ (Corollary \ref{cor:generationAndEigenhull}). 
An eigenset is filtered if and only if it is associated to an odometer (Theorem \ref{the:characterizationOdometers}). 
An eigenset is countable if and only if the associated subodometer is metrizable (Theorem \ref{the:characterizationOdometers}).  

\begin{remark}
    A subset $E\subseteq \subfin(G)$ is a filtered eigenset if and only if it is a conjugation invariant filter. Thus the odometers (up to conjugacy) can be identified with conjugation invariant filters. 
\end{remark}

Note that $\allFilteredEigensets{G}$ and $\allEigensets{G}$ have $\subfin(G)$ as a maximal and $\{G\}$ as a minimal element. We will next study further properties of these partially ordered sets in order to gain insights into the categories of odometers and subodometers, respectively.

\subsection{$\allFilteredEigensets{G}$ as a complete lattice}
\label{subsec:eigensets_FilteredEigensetsAsACompleteLattice}
Next, we show that $\allFilteredEigensets{G}$ is a complete lattice.
We start with the following, which establishes that any non-empty family of odometers allows for a maximal common factor.  

\begin{proposition}
\label{pro:intersectionFiltersAndFilteredEigensets}
\begin{itemize}
    \item[(i)] 
    For any family $\mathcal{S}$ of filters also $\bigcap_{S\in \mathcal{S}} S$ is a filter. 
    \item[(ii)]
    For any family $\mathcal{E}$ of filtered eigensets also $\bigwedge\mathcal{E}:=\bigcap_{E\in \mathcal{E}} E$ is a filtered eigenset. 
\end{itemize}    
\end{proposition}
\begin{proof}
(i): 
    $\check{S}:=\bigcap_{S\in \mathcal{S}}S$ is non-empty, since any filter contains $G$. 
    Furthermore, as the intersection of upward closed sets we have that $\check{S}$ is upward closed. 
    To show that $\check{S}$ is a scale, consider $\Gamma,\Gamma'\in \check{S}$. 
    For $S\in \mathcal{S}$ we have $\Gamma,\Gamma'\in S$ and hence there exists $\Lambda_S\in S$ with $\Lambda_S\subseteq \Gamma\cap \Gamma'$. 
    Consider $\Lambda:=\grouphull{\bigcup_{S\in \mathcal{S}} \Lambda_S}$ and note that $\Lambda\leq G$ is of finite index. 
    For $S\in \mathcal{S}$ we have $\Lambda_S\subseteq \Lambda$ and hence the upward closedness of $S$ yields that $\Lambda\in S$. 
    We thus observe that $\Lambda\in \check{S}$. 
    Furthermore, we have $\Lambda=\grouphull{\bigcup_{S\in \mathcal{S}}\Lambda_S}\subseteq \Gamma \cap \Gamma'$.
    This shows $\check{S}$ to be an upward closed scale, i.e.\ a filter. 

(ii): 
    From (i) we observe that $\bigwedge \mathcal{E}=\bigcap_{E\in \mathcal{E}}E$ is a filter. 
    Furthermore, as the intersection of conjugation invariant sets it is conjugation invariant. 
    We thus know that $\bigwedge \mathcal{E}$ is a conjugation invariant and upward closed scale and observe $\bigwedge \mathcal{E}=\eigenhull{\bigwedge \mathcal{E}}$ to be an eigenset.   
\end{proof}

\begin{corollary}
\label{cor:filteredEigensetsFormCompleteLattice}
    $\allFilteredEigensets{G}$ is a complete lattice. 
\end{corollary}
\begin{proof}
    From Proposition \ref{pro:intersectionFiltersAndFilteredEigensets} we observe that any family in $\allFilteredEigensets{G}$ allows for an infimum. 
    For $\mathcal{E}\subseteq \allFilteredEigensets{G}$ we consider the family $\mathcal{E}'$ of all filtered eigensets that contain $\bigcup_{E\in \mathcal{E}}E$. Clearly, $\bigwedge \mathcal{E}'$ is the supremum of $\mathcal{E}$. 
\end{proof}

\subsection{$\allEigensets{G}$ as a partially ordered set}
\label{subsec:eigensets_EigensetsAsAPartiallyOrderedSet}
Considering Corollary \ref{cor:filteredEigensetsFormCompleteLattice} it is natural to ask whether also $\allEigensets{G}$ is a complete lattice. This is not necessarily the case as illustrated by the next example. Note that this example also yields that for pairs of subodometers there does not need to exist a minimal common extension/maximal common factor. 

\begin{example}
\label{exa:EigensetsNoLattice}
    Let $S_5$ be the group of all permutations on $\{1,2,3,4,5\}$. 
    Denote 
    \[\Lambda_1:=\grouphull{(12)(34)}
    \hspace{0.3cm}\text{and}\hspace{0.3cm}
    \Lambda_2:=\grouphull{(123)}.\]
    Furthermore, consider
    \[\Gamma_1:=\grouphull{(123), (23)(45)}
    \hspace{0.3cm}\text{and}\hspace{0.3cm}
    \Gamma_2:=\grouphull{(123), (12)(34), (13)(24)}.\]
    We will next show that 
    $\eigenhull{\{\Lambda_1\}}$ and $\eigenhull{\{\Lambda_2\}}$
    do not have a supremum and that 
    $\eigenhull{\{\Gamma_1\}}$ and $\eigenhull{\{\Gamma_2\}}$
    do not have an infimum in $\allEigensets{S_5}$. 

    Clearly we have $\Lambda_1\subseteq \Gamma_2$ and $\Lambda_2\subseteq \Gamma_1\cap \Gamma_2$. 
    Furthermore, for $g=(54321)$
    we observe that $\Lambda_1^g=\grouphull{(23)(45)}$ and hence that $\Lambda_1^g\subseteq \Gamma_1$.
    Visualizing the set inclusions of the eigensets via arrows, we thus observe the following diagram. 
    \[
    \begin{tikzcd}[row sep=1.5em, column sep=3em]
    \eigenhull{\{\Lambda_1\}} \arrow[d] \arrow[dr] & \eigenhull{\{\Lambda_2\}} \arrow[d] \arrow[dl] \\
    \eigenhull{\{\Gamma_1\}} & \eigenhull{\{\Gamma_2\}}
    \end{tikzcd}
    \]
    Conjugates of $\Lambda_1$ and $\Lambda_2$ are cyclic groups of order $2$ and $3$, respectively. 
    Since $2$ and $3$ are coprime we observe that $\Lambda_1\not \subseteq \Lambda_2^g$ and 
    $\Lambda_2\not \subseteq \Lambda_1^g$ for all $g\in S_5$. 
    It follows that 
    $\eigenhull{\{\Lambda_1\}}\not \subseteq \eigenhull{\{\Lambda_2\}}$ and 
    $\eigenhull{\{\Lambda_2\}}\not \subseteq \eigenhull{\{\Lambda_1\}}$. 

    All $g\in \Gamma_2$ fix $5$. 
    Thus, if $\Gamma_1$ would be contained in some conjugate of $\Gamma_2$ then there would be $l\in \{1,\dots, 5\}$ that is fixed by all permutations in $\Gamma_1$ and in particular by $(123)$ and $(23)(45)$, a contradiction. 
    Hence, we have $\Gamma_1\not \subseteq \Gamma_2^g$ for all $g\in S_5$. 
    Furthermore, we have $|\Gamma_1|=6$ and $|\Gamma_2|=12$. 
    Thus $\Gamma_2\not \subseteq \Gamma_1^g$ for all $g\in S_5$. 
    It follows that 
    $\eigenhull{\{\Gamma_1\}}\not \subseteq \eigenhull{\{\Gamma_2\}}$
    and that 
    $\eigenhull{\{\Gamma_2\}}\not \subseteq \eigenhull{\{\Gamma_1\}}$. 
    Similarly, we observe that $\eigenhull{\{\Lambda_i\}}\not \subseteq \eigenhull{\{\Gamma_j\}}$ for $i,j\in \{1,2\}$. 
    Summarizing we have shown that the diagram above displays all set inclusions between the considered eigensets.
    
    To observe that $\{\eigenhull{\{\Lambda_1\}},\eigenhull{\{\Lambda_2\}}\}$ has no infimum and that $\{\eigenhull{\{\Gamma_1\}},\eigenhull{\{\Gamma_2\}}\}$ has no supremum in $\allEigensets{S_5}$ we show that there does not exist an intermediate eigenset $E$ with 
    $\eigenhull{\{\Gamma_i\}}\subseteq E\subseteq \eigenhull{\{\Lambda_j\}}$ for $i,j\in \{1,2\}$.    
    To obtain a contradiction, assume that there exists such $E$. 
    Since $S_5$ is finite, any scale that generates $E$ has a minimal element and we thus find $\Gamma\in \subfin(S_5)$ with $E=\eigenhull{\{\Gamma\}}$. 
    
    Since $\Gamma\in E\subseteq \eigenhull{\{\Lambda_1\}}\cap \eigenhull{\{\Lambda_2\}}$ we observe the existence of $g_1,g_2\in S_5$ with $\Lambda_1^{g_1}\cup \Lambda_2^{g_2}\subseteq \Gamma$. 
    Since $\Lambda_1^{g_1}$ and $\Lambda_2^{g_2}$ are cyclic groups of order $2$ and $3$, respectively, we observe that $|\Gamma|\geq 6$. 
    Furthermore, from $\eigenhull{\{\Gamma_1\}}\cup \eigenhull{\{\Gamma_2\}}\subseteq E=\eigenhull{\{\Gamma\}}$ we observe
    $\Gamma_1,\Gamma_2\in \eigenhull{\{\Gamma\}}$. 
    This implies the existence of $h_1,h_2\in S_5$ such that 
    $\Gamma\subseteq \Gamma_1^{h_1}\cap \Gamma_2^{h_2}$. 
    In particular, denoting $h:=h_2h_1^{-1}$ we have $\Gamma^{h_1^{-1}}\subseteq \Gamma_1\cap \Gamma_2^h$. 
    Recall that we have already observed that $\Gamma_1\not \subseteq \Gamma_2^g$ for all $g\in S_5$. 
    We thus have $\Gamma_1\not \subseteq \Gamma_2^h$ and hence 
    $|\Gamma|=|\Gamma^{h_1^{-1}}|\leq |\Gamma_1\cap \Gamma_2^h|<|\Gamma_1|=6$, a contradiction. 
    This shows that an intermediate $E$ as considered cannot exist. 
\end{example}

\subsection{Filter Hulls}
\label{subsec:eigensets_FilterHulls}
From Corollary \ref{cor:filteredEigensetsFormCompleteLattice} we know that any family of odometers allows for a minimal common extension within the category of odometers. 
In order to show that this minimal common extension is actually also the minimal common extension within the category of all subodometers we next develop an explicit formula for the supremum of filtered eigensets. For this we will use the notion of a filter hull. 
Recall from Proposition \ref{pro:intersectionFiltersAndFilteredEigensets} that the intersection of a family of filters is a filter. This allows us to define the following. 

\begin{definition}
    For a subset $S\subseteq \subfin(G)$ we denote $\filterhull{S}$ for the \emph{filter hull of $S$}, i.e.\ the smallest filter that contains $S$. 
\end{definition}

\begin{remark}
    Since any filter is upward closed we have $S\subseteq \upperhull{S}\subseteq \filterhull{S}$. 
    If $S$ is a scale $\upperhull{S}$ is a filter and hence $\filterhull{S}=\upperhull{S}$. 
\end{remark}

The following lemma gives an explicit formula for the filter hull. 
Note that for any finite subset $J\subseteq \subfin(G)$ we have 
$\bigcap_{\Gamma\in J}\Gamma\in \subfin(G)$. 

\begin{lemma}
\label{lem:formulaFilterHull}
    For a subset $S\subseteq \subfin(G)$ we have 
    \[\filterhull{S}=\upperhull{\left\{\bigcap_{\Gamma\in J} \Gamma;\, J\subseteq S \text{ finite}\right\}}.\]
\end{lemma}
\begin{proof}
    Denote $S':=\upperhull{\left\{\bigcap_{\Gamma\in J} \Gamma;\, J\subseteq S \text{ finite}\right\}}$. 
    It is straightforward to verify that $S'$ is a filter. 
    From $S\subseteq S'$ we thus observe $\filterhull{S}\subseteq S'$. 

    For the converse consider $\Gamma\in S'$. 
    There exists a finite subset $J\subseteq S$ with $\bigcap_{\Gamma'\in J} \Gamma'\subseteq \Gamma$. 
    For any filter $\hat{S}$ that contains $S$ we observe $J\subseteq S\subseteq \hat{S}$. 
    Since $\hat{S}$ is a scale and $J$ is finite there exists $\Lambda\in \hat{S}$ with $\Lambda\subseteq \bigcap_{\Gamma'\in J}\Gamma'\subseteq \Gamma$.
    Since $\hat{S}$ is upward closed we observe $\Lambda\in \hat{S}$. 
    This shows that $\Lambda$ is contained in all filters $\hat{S}$ that contain $S$ and hence that $\Lambda\in \filterhull{S}$. 
\end{proof}

\begin{lemma}
\label{lem:filterHullProperties}
    Let $S\subseteq \subfin(G)$ be a subset. 
    \begin{itemize}
        \item[(i)] If $S$ is finite, then $\filterhull{S}$ is finite. 
        \item[(ii)] If $S$ is countable, then $\filterhull{S}$ is countable. 
        \item[(iii)] If $S$ is conjugation invariant, then $\filterhull{S}$ is a filtered eigenset. 
        \item[(iv)] If $S$ is a scale and $E$ is a filtered eigenset with $E\subseteq \eigenhull{S}$, then $E\subseteq \filterhull{S}$. 
    \end{itemize}
\end{lemma}
\begin{proof}
(i/ii): 
    For $J\subseteq S$ finite we have $\bigcap_{\Gamma\in J}\Gamma\in \subfin(G)$. Note that $\upperhull{\{\bigcap_{\Gamma\in J}\Gamma\}}$ is finite by Corollary \ref{cor:eigenhullFinite}.  
    Whenever $S$ is finite/countable we thus observe that
    \[\filterhull{S}=\upperhull{\left\{\bigcap_{\Gamma\in J}\Gamma;\, J\subseteq S \text{~finite}\right\}}
    =\bigcup_{J\subseteq S \text{~finite}}
    \upperhull{\left\{\bigcap_{\Gamma\in J}\Gamma\right\}}
    \]
    is a finite/countable union of finite sets and hence finite/countable. 

(iii):     
    Assume that $S$ is conjugation invariant. 
    For $J\subseteq S$ finite and $g\in G$ we have $J':=\{\Gamma^g;\, \Gamma\in J\}\subseteq S$ and that 
    $(\bigcap_{\Gamma\in J}\Gamma)^g=\bigcap_{\Gamma\in J}\Gamma^g=\bigcap_{\Gamma\in J'}\Gamma$. 
    This shows $\{\bigcap_{\Gamma\in J}\Gamma;\, J\subseteq S \text{ finite}\}$ to be conjugation invariant. 
    Since upper hulls of conjugation invariant sets are conjugation invariant we observe from Lemma \ref{lem:formulaFilterHull} that $\filterhull{S}$ is conjugation invariant.

    Since $\filterhull{S}$ is a conjugation invariant and upward closed scale we have $\filterhull{S}=\eigenhull{\filterhull{S}}$ and hence $\filterhull{S}$ is an eigenset. 

(iv): 
    Let $\Gamma\in E$ and note that $\Gamma$ only has finitely many conjugates. 
    Since $E$ is a conjugation invariant scale we find $\Gamma'\in E$ with $\Gamma'\subseteq \bigcap_{g\in G}\Gamma^g=\Gamma_G$. 
    Since $E$ is upward closed we observe $\Gamma_G\in E\subseteq \eigenhull{S}$.
    Thus there exist $g\in G$ and $\Lambda\in S$ with $\Lambda\subseteq \Gamma_G^g=\Gamma_G\subseteq \Gamma$. 
    Hence, $\Gamma\in \upperhull{S}=\filterhull{S}$. 
\end{proof}

\subsection{$\allFilteredEigensets{G}$ as a complete sublattice of $\allEigensets{G}$}
\label{subsec:eigensets_FilteredEigensetsAsACompleteSublatticeOfEigensets}

We next present an explicit formula for the supremum of filtered eigensets. 

\begin{proposition}
\label{pro:explicitSupremum}
    Let $\mathcal{E}$ be a family of filtered eigensets and denote 
     \[\bigvee \mathcal{E}:= \filterhull{\left(\bigcup_{E\in\mathcal{E}}E\right)}.\]
    \begin{itemize}
        \item[(i)] $\bigvee\mathcal{E}$ is a filtered eigenset. 
        \item[(ii)] All eigensets $\hat{E}$ that contain $\bigcup_{E\in \mathcal{E}}E$ also contain $\bigvee \mathcal{E}$. 
    \end{itemize}
\end{proposition}
\begin{proof}
(i): 
    Since the union of conjugation invariant sets is conjugation invariant we observe $\bigcup_{E\in \mathcal{E}}E$ to be conjugation invariant. It follows from Lemma \ref{lem:filterHullProperties} that $\bigvee \mathcal{E}$ is a filtered eigenset. 

(ii): 
    Let $\hat{E}$ be an eigenset that contains $\bigcup_{E\in \mathcal{E}}E$ and choose a scale $S$ with $\hat{E}=\eigenhull{S}$.
    From Lemma \ref{lem:filterHullProperties} we observe that for $E\in \mathcal{E}$ we have $E\subseteq \filterhull{S}$, i.e.\ $\bigcup_{E\in \mathcal{E}}E\subseteq \filterhull{S}$. 
    It follows that $\bigvee \mathcal{E}=\filterhull{(\bigcup_{E\in \mathcal{E}}E)}\subseteq \filterhull{S}$.
    Since $S$ is a scale we have $\filterhull{S}=\upperhull{S}$ and hence 
    $\bigvee \mathcal{E}\subseteq \filterhull{S}\subseteq \conhull{(\upperhull{S})}=\eigenhull{S}=\hat{E}$. 
\end{proof}

We can now strengthen the statement of Corollary \ref{cor:filteredEigensetsFormCompleteLattice}. 

\begin{theorem}
\label{the:completeSublatticeEigensets}
    $\allFilteredEigensets{G}$ is a complete sublattice of $\allEigensets{G}$, i.e.\ for any family of filtered eigensets $\mathcal{E}$ there exists a supremum $\bigvee \mathcal{E}$ and an infimum $\bigwedge \mathcal{E}$ in $\allEigensets{G}$ and both are contained in $\allFilteredEigensets{G}$.  
\end{theorem}
\begin{remark}
    The formulas for the supremum and infimum are given by 
    $\bigvee \mathcal{E}:= \filterhull{(\bigcup_{E\in \mathcal{E}}E)}$ and $\bigwedge \mathcal{E}:=\bigcap_{E\in \mathcal{E}}E$. 
\end{remark}
\begin{proof}[Proof of Theorem \ref{the:completeSublatticeEigensets}.]
    From Proposition \ref{pro:intersectionFiltersAndFilteredEigensets} we know that $\bigwedge\mathcal{E}:=\bigcap_{E\in \mathcal{E}}E$ is a filtered eigenset. 
    Clearly, it is not only the infimum of $\mathcal{E}$ in $\allFilteredEigensets{G}$, but also in $\allEigensets{G}$. 

    It follows from Proposition \ref{pro:explicitSupremum} that 
    $\bigvee \mathcal{E}:= \filterhull{(\bigcup_{E\in \mathcal{E}}E)}$ is a filtered eigenset and the supremum of $\mathcal{E}$ within $\allEigensets{G}$. 
\end{proof}

\begin{proof}[Proof of Theorem \ref{the:INTROOdometerLattice}:]
    Let $\mathfrak{X}$ be a non-empty family of odometers. 
    Denote $\bigvee\mathfrak{X}$ for the odometer associated with $\bigvee \{\Eig(X,G);\, (X,G)\in \mathfrak{X}\}$ and 
    $\bigwedge\mathfrak{X}$ for the odometer associated with $\bigwedge \{\Eig(X,G);\, (X,G)\in \mathfrak{X}\}$. 
    It follows from the discussion at the beginning of this section that $\bigvee \mathfrak{X}$ and $\bigwedge \mathfrak{X}$ satisfy (a$_1$) and (a$_2$). 
    Furthermore, since $\mathfrak{X}$ is non-empty and factors of subodometers are subodometers, we observe that $\bigwedge \mathfrak{X}$ satisfies (b$_2$) from Theorem \ref{the:completeSublatticeEigensets}. 

    To show that $\bigvee \mathfrak{X}$ satisfies (b$_1$) consider a minimal action $(Y,G)$ that is an extension of all $(X,G)\in \mathfrak{X}$. 
    From Theorem \ref{the:eigConjugacyInvariant} we know that 
    $\Eig(X,G)\subseteq \Eig(Y,G)$ for all $(X,G)\in \mathfrak{X}$. 
    Since $\Eig(Y,G)$ is an eigenset (Remark \ref{rem:eigensetsViaMinimalActions}), it follows that 
    \[\Eig\left(\bigvee\mathfrak{X},G\right)=\bigvee \{\Eig(X,G);\, (X,G)\in \mathfrak{X}\}\subseteq \Eig(Y,G).\] 
    Thus, Theorem \ref{the:eigConjugacyInvariant} yields that $\bigvee \mathfrak{X}$ is a factor of $(Y,G)$. 

    To show that $\bigvee \mathfrak{X}$ is the unique minimal action (up to conjugacy) that satisfies (a$_1$) and (b$_1$) consider a minimal action $(Z,G)$ that also satisfies (a$_1$) and (b$_1$). 
    It follows that $(Z,G)$ is an extension and a factor of $\bigvee \mathfrak{X}$. 
    As a factor of $\bigvee\mathfrak{X}$ it is a subodometer and the coalescence of minimal equicontinuous actions yields that $\bigvee\mathfrak{X}$ and $(Z,G)$ are conjugated. 
    Similarly, one establishes that $\bigwedge \mathfrak{X}$ is the unique minimal action (up to conjugacy) that satisfies (a$_2$) and (b$_2$). 
\end{proof}

\subsection{Modularity}
\label{subsec:eigensets_Modularity}
It is well-known that the set $\subfinnorm(G)$ of all normal finite index subgroups of $G$ forms a lattice under set inclusion, where the supremum and infimum are given by 
$\grouphull{\Gamma\cup \Lambda}=\Gamma\Lambda$ and 
$\Gamma\cap \Lambda$
\cite[Chapter 1]{birkhoff1967lattice}. 
As a lattice $\subfinnorm(G)$ is \emph{modular}, i.e.\ for $\Gamma_1,\Gamma_2,\Gamma\in \subfin(G)$ with $\Gamma_1\supseteq \Gamma_2$ the \emph{modular law}
\[(\Gamma_1\cap \Gamma)\Gamma_2=\Gamma_1\cap (\Gamma\Gamma_2)\]
is satisfied. 
Reversing the order of $\subfinnorm(G)$ it embeds into the order complete lattice $\allFilteredEigensets{G}$ via $\Gamma\mapsto \eigenhull{\{\Gamma\}}$. 
We will next show that $\allFilteredEigensets{G}$ inherits the modularity of $\subfin(G)$, i.e.\ we show Theorem \ref{the:INTROmodularLattice}. 

\begin{proposition}[Modular Law]
    For $E_1,E_2,E\in \allFilteredEigensets{G}$ with $E_1\subseteq E_2$ we have 
    $(E_1\vee E)\wedge E_2=E_1\vee (E\wedge E_2).$
\end{proposition}
\begin{proof}
    Clearly, we have $E_1\cup (E\cap E_2)\subseteq (E_1\cup E)\cap E_2\subseteq \filterhull{(E_1\cup E)}\cap E_2$. 
    The latter is an intersection of filters and hence a filter. 
    Thus, we have
    \[E_1\vee (E\wedge E_2)=\filterhull{(E_1\cup (E\cap E_2))}\subseteq \filterhull{(E_1\cup E)}\cap E_2=(E_1\vee E)\wedge E_2.\] 

    For the converse consider $\Lambda\in (E_1\vee E)\wedge E_2=\filterhull{(E_1\cup E)}\cap E_2$. 
    By Lemma \ref{lem:formulaFilterHull} there exists a finite subset $J\subseteq E_1\cup E$ with $\bigcap_{\Gamma\in J}\Gamma\subseteq \Lambda$. 
    Since $E_1$ and $E$ are filtered eigensets they are scales and core-stable (Theorem \ref{the:characterizationOdometers}). 
    We thus assume w.l.o.g.\ that $J=\{\Gamma_1,\Gamma\}$ for some normal subgroups $\Gamma_1\in E_1$ and $\Gamma\in E$.
    Note that $\Gamma_1\cap \Gamma\subseteq \Lambda$. 

    Furthermore, since $\Gamma_1\in E_1\subseteq E_2$ we have $\Gamma_1,\Lambda\in E_2$. Since $E_2$ is a core-stable scale we can choose a normal $\Gamma_2\in E_2$ with $\Gamma_2\subseteq \Gamma_1\cap \Lambda$. We have $\Gamma\Gamma_2\supseteq \Gamma\in E$ and $\Gamma\Gamma_2\supseteq \Gamma_2\in E_2$ and hence $\Gamma\Gamma_2\in E\cap E_2$. 
    Denote $J':=\{\Gamma_1,\Gamma\Gamma_2\}$ and note that $J'\subseteq E_1\cup (E\cap E_2)$. 
    Furthermore, exploring the modular law for $\Gamma_1\supseteq \Gamma_2$ we have 
    \begin{align*}
        \bigcap_{\Gamma'\in J'}\Gamma'
        &=\Gamma_1\cap (\Gamma\Gamma_2)
        =(\Gamma_1\cap \Gamma)\Gamma_2
        \subseteq \Lambda\Lambda
        =\Lambda.
    \end{align*}
    From the formula of the filter hull presented in Lemma \ref{lem:formulaFilterHull} we conclude that
    $\Lambda\in \filterhull{(E_1\cup (E\cap E_2))}= E_1\vee (E\wedge E_2)$. 
\end{proof}

\subsection{Suprema of metrizable odometers}
\label{subsec:eigensets_SupremaOfMetrizableOdometers}

Recall that the metrizability of a subodometer is reflected in the countability of the respective eigenset. 
Whenever $\mathfrak{X}$ is a non-empty family of metrizable odometers, then clearly also its factor $\bigwedge\mathfrak{X}$ is metrizable. 
In general the minimal common extension $\bigvee \mathfrak{X}$ of a non-empty family of metrizable odometers $\mathfrak{X}$ needs not be metrizable as illustrated by the following. 

\begin{example}
    Let $G$ be a group for which $S:=\subfin(G)$ is uncountable (e.g.\ $G=F_\infty$) and note that $S$ is a scale. 
    Clearly, the $S$-subodometer $(X,G)$ is not metrizable. 
    However, it is straightforward to observe that $(X,G)$ is the minimal common extension of all finite odometers, which are clearly metrizable. 
\end{example}

Translating the following back to the category of odometers yields Theorem \ref{the:INTROmetrizabilitySupremum}.

\begin{proposition}
    If $\mathcal{E}$ is a countable family of countable filtered eigensets, then $\bigvee \mathcal{E}$ is also countable. 
\end{proposition}
\begin{proof}
    Note that $\bigcup_{E\in \mathcal{E}}E$ is countable as the countable union of countable sets. 
    It follows from the Lemmas \ref{lem:formulaFilterHull} and \ref{lem:filterHullProperties} that also 
    $\bigvee \mathcal{E}=\filterhull{(\bigcup_{E\in \mathcal{E}}E)}$ is countable. 
\end{proof}

\subsection{Scales for the minimal common extension and the maximal common factor}
\label{subsec:eigensets_ScalesForTheMinimalCommonExtensionAndTheMaximalCommonFactors}

\begin{proposition}
\label{pro:supremumInfimumScales}
    Let $\mathfrak{X}$ be a non-empty family of odometers. For $X\in \mathfrak{X}$ let $S_X$ be a scale that generates $X$. 
    \begin{itemize}
        \item[(i)] The minimal common extension $\bigvee \mathfrak{X}$ is generated by the scale
        \[
            S:=\left\{\bigcap_{\Gamma\in J}\Gamma;\, J\subseteq \bigcup_{X\in \mathfrak{X}}S_X\text{ finite}\right\}
        \] 
        \item[(ii)] The maximal common factor $\bigwedge \mathfrak{X}$ is generated by the scale 
        \[
            S:=\left\{\grouphull{\bigcup_{X\in \mathfrak{X}}\Gamma_X};\, (\Gamma_X)_{X\in \mathfrak{X}}\in \prod_{X\in \mathfrak{X}}S_X\right\}
        \]
    \end{itemize}
\end{proposition}
\begin{proof}
(i):
    It is straightforward to observe that $S$ is a scale.
    For $X\in \mathfrak{X}$ we have $S_X\subseteq S$ and hence $\eigenhull{S_X}\subseteq \eigenhull{S}$. 
    It follows that 
    \[\bigvee\{\Eig(X,G);\, X\in \mathfrak{X}\}=\bigvee\{\eigenhull{S_X};\, X\in \mathfrak{X}\}\subseteq \eigenhull{S}.\]

    Furthermore, it follows from Lemma \ref{lem:formulaFilterHull} and $\bigcup_{X\in \mathfrak{X}}S_X\subseteq \bigcup_{X\in \mathfrak{X}}\eigenhull{S_X}$ that 
    \[\upperhull{S}=\filterhull{S}
    \subseteq \filterhull{\left(\bigcup_{X\in \mathfrak{X}}\eigenhull{S_X}\right)}
    =\bigvee\{\eigenhull{S_X};\, X\in \mathfrak{X}\}.\] 
    Since $\bigvee\{\eigenhull{S_X};\, X\in \mathfrak{X}\}$ is conjugation invariant we observe that 
    \[\eigenhull{S}=\conhull{(\upperhull{S})}\subseteq \bigvee\{\eigenhull{S_X};\, X\in \mathfrak{X}\}=\bigvee\{\Eig(X,G);\, X\in \mathfrak{X}\}.\]
    This shows $\eigenhull{S}=\bigvee\{\Eig(X,G);\, X\in \mathfrak{X}\}$ and we conclude that $S$ generates $\bigvee \mathfrak{X}$. 

(ii): 
    To show that $S$ is a scale consider $(\Gamma_X)_X,(\Gamma_X')_X\in \prod_{X\in \mathfrak{X}}S_X$. 
    For $X\in \mathfrak{X}$ there exists $\Lambda_X\in S_X$ with $\Lambda_X\subseteq \Gamma_X\cap \Gamma_X'$. 
    Clearly, we have $S\ni \grouphull{\bigcup_{X\in \mathfrak{X}}\Lambda_X}\subseteq \grouphull{\bigcup_{X\in \mathfrak{X}}\Gamma_X}\cap \grouphull{\bigcup_{X\in \mathfrak{X}}\Gamma_X'}$. 
    This shows that $S$ is a scale. 

    For $X\in \mathfrak{X}$ it is straightforward to observe that $S_X$ dominates $S$. 
    This establishes 
    \[\eigenhull{S}\subseteq \bigcap_{X\in \mathfrak{X}}\eigenhull{S_X} =\bigwedge\{\Eig(X,G);\, X\in \mathfrak{X}\}.\]

    For the converse consider $\Gamma\in \bigwedge\{\Eig(X,G);\, X\in \mathfrak{X}\}=\bigcap_{X\in \mathfrak{X}}\eigenhull{S_X}$.
    Let $X\in \mathfrak{X}$. 
    Since $\eigenhull{S_X}$ is a filtered eigenset and $S_X$ is a scale we observe from Lemma \ref{lem:filterHullProperties} that $\Gamma\in\eigenhull{S_X}\subseteq \filterhull{S_X}=\upperhull{S_X}$.
    Thus there exists $\Gamma_X\in X$ with $\Gamma_X\subseteq \Gamma$. 
    We thus have $S\ni \grouphull{\bigcup_{X\in \mathfrak{X}}\Gamma_X}\subseteq \Gamma$ and hence $\Gamma\in \upperhull{S}\subseteq \eigenhull{S}$.   

    This shows $\eigenhull{S}=\bigwedge\{\Eig(X,G);\, X\in \mathfrak{X}\}$ and we conclude that $S$ generates $\bigwedge \mathfrak{X}$. 
\end{proof}

Recall that metrizable odometers are generated by normal chains and that any chain can be represented by a sequence $(\Gamma_n)_{n\in \mathbb{N}}$ with $\Gamma_{n+1}\subseteq \Gamma_n$. 
In the following whenever we represent a chain as a sequence we implicitly assume that $\Gamma_{n+1}\subseteq \Gamma_n$ is satisfied. 

\begin{proposition}
\label{pro:supremumChain}
    Let $\mathfrak{X}=(X_i)_{i\in \mathbb{N}}$ be a sequence of odometers.
    For $i\in \mathbb{N}$ let $(\Gamma_n^{(i)})_{n\in \mathbb{N}}$ be a chain that generates $X_i$ and denote
    \[\Gamma_n:=\bigcap_{i=1}^n\Gamma_n^{(i)}\]
    for $n\in \mathbb{N}$. 
    The chain $(\Gamma_n)_{n\in \mathbb{N}}$ generates the minimal common extension $\bigvee \mathfrak{X}$. 
\end{proposition}
\begin{remark}
    Note that if all $(\Gamma_n^{(i)})_{n\in \mathbb{N}}$ are normal, so is $(\Gamma_n)_{n\in \mathbb{N}}$. 
\end{remark}
\begin{remark}
    It follows that for normal chains $(\Gamma_n)_{n\in \mathbb{N}}$ and $(\Lambda_n)_{n\in \mathbb{N}}$ the minimal common extension of the generated odometers is generated by $(\Gamma_n\cap \Lambda_n)_{n\in \mathbb{N}}$. 
\end{remark}
\begin{proof}[Proof of Proposition \ref{pro:supremumChain}:]
    Clearly, $(\Gamma_n)_{n\in \mathbb{N}}$ is a chain. 
    Denote $S$ for the scale given in Proposition \ref{pro:supremumInfimumScales} and note that $\Gamma_n\in S$ holds for all $n\in \mathbb{N}$. 
    In order to show that $(\Gamma_n)_{n\in \mathbb{N}}$ dominates $S$ consider $\Gamma\in S$ and choose a finite subset $J\subseteq \bigcap_{i,n\in \mathbb{N}}\Gamma_n^{(i)}$ with $\Gamma=\bigcap_{\Lambda\in J}\Lambda$. 
    For $\Lambda\in J$ there exist $i_\Lambda,n_\Lambda\in \mathbb{N}$ such that $\Lambda=\Gamma_{n_\Lambda}^{(i_\Lambda)}$. 
    Let $n:=\max\{i_\Lambda,n_\Lambda;\, \Lambda\in J\}$ and note that 
    \[\Gamma_n
    =\bigcap_{i=1}^n \Gamma_n^{(i)}  
    =\bigcap_{i=1}^n \bigcap_{k=1}^n \Gamma_k^{(i)} 
    \subseteq \bigcap_{\Lambda\in J} \Gamma_{n_\Lambda}^{(i_\Lambda)}
    =\bigcap_{\Lambda\in J}\Lambda
    =\Gamma.\]
    This shows that $(\Gamma_n)_{n\in \mathbb{N}}$ and $S$ are equivalent scales and Proposition \ref{pro:supremumInfimumScales} yields
    that both generate $\bigvee \mathfrak{X}$. 
\end{proof}

Given a countable (non-empty) family of normal chains it remains open, whether the formula in Proposition \ref{pro:supremumInfimumScales} yields a simplification similar to Proposition \ref{pro:supremumChain}. 
However, for pairs of metrizable odometers, we have the following. 

\begin{proposition}
    Let $(X,G)$ and $(Y,G)$ be subodometers generated by chains $(\Gamma_n)_{n\in \mathbb{N}}$ and $(\Lambda_n)_{n\in \mathbb{N}}$. 
    The maximal common factor $X\wedge Y$ is generated by the chain $(\grouphull{\Gamma_n\cup \Lambda_n})_{n\in \mathbb{N}}$. 
\end{proposition}
\begin{remark}
    If $(\Gamma_n)_{n\in \mathbb{N}}$ and $(\Lambda_n)_{n\in \mathbb{N}}$ are normal chains, then 
    $(\Gamma_n \Lambda_n)_{n\in \mathbb{N}}=(\grouphull{\Gamma_n\cup \Lambda_n})_{n\in \mathbb{N}}$ is a normal chain that generates $X\wedge Y$. 
\end{remark}
\begin{proof}
    Denote $S$ for the scale considered in Proposition \ref{pro:supremumInfimumScales} and note that $\Gamma_n\in S$ holds for all $n\in \mathbb{N}$. Thus $S$ dominates $(\Gamma_n)_{n\in \mathbb{N}}$. 
    
    To show the converse consider $\Gamma\in S$. There exist $(n,m)\in \mathbb{N}^2$ with $\Gamma=\Gamma_n\Lambda_m$. 
    W.l.o.g.\ we assume that $n\geq m$ and observe $S\ni\Gamma_n\Lambda_n\subseteq \Gamma_n\Lambda_m=\Gamma$. 
    This shows that $(\Gamma_n)_{n\in \mathbb{N}}$ dominates $S$. 
    Thus, $S$ and $(\Gamma_n)_{n\in \mathbb{N}}$ are equivalent scales and Proposition \ref{pro:supremumInfimumScales} yields that both generate $X\wedge Y$. 
\end{proof}

\section{Universal subodometers}
\label{sec:universalSubodometers}

\subsection{The universal odometer}
\label{subsec:universalSubodometers_theUniversalOdometer}

Note that $\subfin(G)$ is a filtered eigenset. 
Considering the associated odometer we observe the following. 

\begin{proposition}
    There exists an odometer $(\hat{X},G)$ that has all subodometers $(X,G)$ as factors. It is unique up to conjugacy. 
    The odometer $(\hat{X},G)$ is called the \emph{universal odometer (w.r.t.\ $G$)}. 
\end{proposition}

\begin{remark}
    The universal odometer is metrizable if and only if $\subfin(G)$ is countable, as already discussed in Corollary \ref{cor:allSubodometersMetrizable}. 
\end{remark}

\begin{remark}
    The set $S$ of all normal $\Gamma\in \subfin(G)$ is intersection closed and hence a scale. 
    For $\Lambda\in \subfin(G)$ we have $\Lambda_G\leq\Lambda$. 
    Thus, $S$ dominates $\subfin(G)$ and we observe that $S$ generates the universal odometer. 
\end{remark}

\subsection{The maximal subodometer factor}
\label{subsec:universalSubodometers_theMaximalSubodometerFactor}

Let $(X,G)$ be a minimal action. As discussed in Remark \ref{rem:eigensetsViaMinimalActions} $\Eig(X,G)$ is an eigenset. 
Considering the associated subodometer Theorem \ref{the:eigConjugacyInvariant} yields the following. 

\begin{proposition}
\label{pro:maximalSubodometerFactorExistence}
    Let $(X,G)$ be a minimal action. 
    There exists a subodometer $(X_{\operatorname{MSF}},G)$ such that 
    \begin{itemize}
        \item[(a)] $(X_{\operatorname{MSF}},G)$ is a factor of $(X,G)$. 
        \item[(b)] Whenever $(Y,G)$ is a subodometer and a factor of $(X,G)$, then $(Y,G)$ is a factor of $(X_{\operatorname{MSF}},G)$. 
    \end{itemize}
    The subodometer satisfying (a) and (b) is unique up to conjugacy and called the \emph{maximal subodometer factor of $(X,G)$}. 
    It satisfies 
    \[\Eig(X,G)=\Eig(X_{\operatorname{MSF}},G).\] 
\end{proposition}

\begin{remark}
    The maximal subodometer factor of $X$ does not need to be an odometer. It is an odometer, if and only if $\Eig(X,G)$ is a filter. 
\end{remark}

\begin{remark}
    If $(X,G)$ is metrizable, then so is the maximal subodometer factor and we observe $\Eig(X,G)$ to be countable. This shows that metrizable actions have only countably many finite factors up to conjugacy. 
\end{remark}

\begin{remark}
    The universal subodometer is the maximal subodometer factor of the universal minimal flow \cite[Theorem 8.1]{auslander1988minimal}. 
    The maximal subodometer factor of a minimal action $(X,G)$ is conjugated to the maximal subodometer factor of the maximal equicontinuous factor $(X,G)$ \cite[Remark 9.4]{hauser2025mean}. 
\end{remark}

\begin{remark}
    Any scale $S$ that generates $\Eig(X,G)$ also generates the maximal subodometer factor of $(X,G)$. 
    In particular, it follows from Theorem \ref{the:scalesForSubodometersViaEigx} that for $x\in X$ the filter $\Eig_x(X,G)$ generates $(X_{\operatorname{MSF}},G)$. 
\end{remark}

\subsection{The maximal odometer factor}
\label{subsec:universalSubodometers_theMaximalOdometerFactor}
For a minimal action $(X,G)$ denote $\mathfrak{X}$ for the set of all odometers that are factors of $(X,G)$. 
Since the trivial action is an odometer we observe $\mathfrak{X}$ to be non-empty. 
Considering the minimal common extension $\bigvee \mathfrak{X}$ we observe the following.

\begin{proposition}
\label{pro:maximalOdometerFactorExistence}
    Let $(X,G)$ be a minimal action. 
    There exists an odometer $(X_{\operatorname{MOF}},G)$ such that 
    \begin{itemize}
        \item[(a)] $(X_{\operatorname{MOF}},G)$ is a factor of $(X,G)$. 
        \item[(b)] Whenever $(Y,G)$ is an odometer and a factor of $(X,G)$, then $(Y,G)$ is a factor of $(X_{\operatorname{MOF}},G)$. 
    \end{itemize}
    The odometer satisfying (a) and (b) is unique up to conjugacy and called the \emph{maximal odometer factor of $(X,G)$}. 
\end{proposition}

\begin{remark}
    Note that for a subodometer $(X,G)$ that is not an odometer the maximal subodometer factor is $(X,G)$ and differs from the maximal odometer factor. 
    It follows from \cite[Remark 9.4]{hauser2025mean} that the maximal odometer factor of a minimal action $(X,G)$ is the maximal odometer factor of the maximal subodometer factor. 
\end{remark}

Recall that the maximal subodometer factor of a minimal action $(X,G)$ has the same eigenvalues as $(X,G)$. 
We next show how to compute the eigenvalues of the maximal odometer factor. 

\begin{proposition}
    \label{pro:maximalOdometerFactorEigenvalues}
    Let $(X,G)$ be a minimal action and $(X_{\operatorname{MOF}},G)$ be its maximal odometer factor. 
    We have $\Eig(X_{\operatorname{MOF}},G)=\bigcap_{x\in X}\Eig_x(X,G)$. 
\end{proposition}
\begin{proof}
    Recall from Theorem \ref{the:scalesForSubodometersViaEigx} that $\Eig_x(X,G)$ is a filter for $x\in X$. 
    Furthermore, from Proposition \ref{pro:intersectionFiltersAndFilteredEigensets} we know that the intersection of filters is a filter. 
    Thus, $E:=\bigcap_{x\in X}\Eig_x(X,G)$ is a filter.
    It is straightforward to verify that $E$ is conjugation invariant. 
    We thus observe that $E$ is a filtered eigenset. 
    Denote $(X',G)$ for the $E$-odometer. From $E\subseteq \Eig(X,G)$ we observe that $(X',G)$ is a factor of $(X,G)$ and hence that $E=\Eig(X',G)\subseteq \Eig(X_{\operatorname{MOF}},G)$. 

    For the converse, consider a normal scale $S$ with $\eigenhull{S}=\Eig(X_{\operatorname{MOF}},G)$. 
    For $\Gamma\in S$ there exists a factor map $X\to G/\Gamma$. 
    Since $(X_{\operatorname{MOF}},G)$ is a factor of $(X,G)$ there exists a factor map $\pi\colon X\to G/\Gamma$ and by the normality of $\Gamma$ for $x\in X$ we can modify $\pi$ to enforce $\pi(x)=\Gamma$. 
    We thus have $\Gamma\in \bigcap_{x\in X}\Eig_x(X,G)=E$. 
    This shows $S\subseteq E$. 
    Since $E$ is a filtered eigenset we have $\Eig(X_{\operatorname{MOF}},G)=\eigenhull{S}\subseteq E$. 
\end{proof}

Whenever $(X,G)$ is a subodometer generated by a scale $S$ it is natural to ask how $S$ can be modified to obtain a scale for the maximal odometer factor of $(X,G)$. 
For $\Gamma\in \subfin(G)$ we denote 
\[\Gamma^G:=\grouphull{\bigcup_{g\in G}\Gamma^g}\]
for the \emph{normal hull} of $\Gamma$. 
Note that the normal hull is the smallest normal subgroup of $G$ that contains $\Gamma$. 

\begin{proposition}
    If $(X,G)$ is a subodometer and $S$ a scale that generates $(X,G)$, then $S^G:=\{\Gamma^G;\, \Gamma\in S\}$ is a normal scale that generates $(X_{\operatorname{MOF}},G)$.
\end{proposition}
\begin{proof}
    Clearly, $S^G$ is normal. 
    Since the map $S\ni \Gamma \mapsto \Gamma^G$ is monotone w.r.t.\ set inclusion we observe that $S^G$ is a scale. 
    Denote $(X^G,G)$ for the odometer generated by $S^G$. 

    It follows from $\Gamma\subseteq \Gamma^G$ for all $\Gamma\in S$ that $S$ dominates $S^G$. Thus, 
    $(X^G,G)$ is a factor of $(X,G)$. 

    Consider a factor $(Y,G)$ of $(X,G)$ that is an odometer. 
    Since $\Eig(Y,G)$ is a filtered eigenset contained in $\Eig(X,G)=\eigenhull{S}$ we observe from Lemma \ref{lem:filterHullProperties} that $\Eig(Y,G)\subseteq \filterhull{S}=\upperhull{S}$.     
    For $\Gamma\in \Eig(Y,G)$ we observe from the core-stability of $\Eig(Y,G)$ that $\Gamma_G\in \Eig(Y,G)\subseteq \upperhull{S}$.
    Thus, there exists $\Lambda\in S$ with $\Lambda\subseteq \Gamma_G$.
    Since $\Gamma_G$ is normal it follows that $\Lambda^G\subseteq (\Gamma_G)^G=\Gamma_G\subseteq \Gamma$. 
    We thus have $\Gamma\in \eigenhull{S^G}=\Eig(X^G,G)$. 
    This shows $\Eig(Y,G)\subseteq \Eig(X^G,G)$ and hence that $(Y,G)$ is a factor of $(X^G,G)$. 

    Thus $(X^G,G)$ satisfies the properties (a) and (b) of Proposition \ref{pro:maximalOdometerFactorExistence} and establishes $(X^G,G)$ as the maximal odometer factor of $(X,G)$. 
\end{proof}

\subsection{The enveloping odometer}
\label{subsec:universalSubodometers_theEnvelopingOdometer}

Consider a subodometer $(X,G)$ and denote $\mathfrak{X}$ for the family of all odometers that are extensions of $(X,G)$. 
$\mathfrak{X}$ is non-empty, since it contains the universal odometer. 
Considering $\bigwedge\mathfrak{X}$ we observe the following from Theorem \ref{the:INTROOdometerLattice}. 

\begin{proposition}
\label{pro:envelopingOdometerExistence}
    Let $(X,G)$ be a subodometer. 
    There exists an odometer $(\envelopingOdometer{X},G)$ such that 
    \begin{itemize}
        \item[(a)] $(X,G)$ is a factor of $(\envelopingOdometer{X},G)$ and 
        \item[(b)] Any odometer $(X',G)$ that is an extension of $(X,G)$ is also an extension of $(\envelopingOdometer{X},G)$. 
    \end{itemize}
    The odometer satisfying (a) and (b) is unique up to conjugacy and called the \emph{enveloping odometer}. 
\end{proposition}

A scale of the enveloping odometer is given as follows. 

\begin{proposition}
\label{pro:envelopingOdometerScale}
    If $(X,G)$ is a subodometer and $S$ a scale that generates $(X,G)$, then $S_G:=\{\Gamma_G;\, \Gamma\in S\}$ is a normal scale that generates $(\envelopingOdometer{X},G)$. 
\end{proposition}
\begin{proof}
    Clearly, $S_G$ is normal. 
    Since the map $S\ni \Gamma\mapsto \Gamma_G$ is monotone w.r.t.\ set inclusion, we observe that $S_G$ is a scale.
    Denote $(X_G,G)$ for the $S_G$-odometer. 
    We have $\Gamma_G\subseteq \Gamma$ for all $\Gamma\in S$ and hence $S_G$ dominates $S$. Thus $(X,G)$ is a factor of $(X_G,G)$. 

    Let $(X',G)$ be an extension of $(X,G)$ that is an odometer. 
    It follows from Theorem \ref{the:eigConjugacyInvariant} that $S\subseteq \Eig(X,G)\subseteq \Eig(X',G)$. 
    From Theorem \ref{the:characterizationOdometers} we know that $\Eig(X',G)$ is a core-stable filter and hence we have $S_G\subseteq \Eig(X',G)$. 
    Since $\Eig(X',G)$ is conjugation invariant and upward closed we have
    $\eigenhull{S_G}\subseteq \Eig(X',G)$ and Proposition \ref{pro:scaleEigensetTransition} yields $\Eig(\envelopingOdometer{X},G)\subseteq \Eig(X',G)$. 
    From Theorem \ref{the:eigConjugacyInvariant} we observe that $(X',G)$ is an extension of $(X,G)$. 

    This shows that $(X_G,G)$ satisfies (a) and (b) of Proposition \ref{pro:envelopingOdometerExistence} and hence is conjugated to the enveloping odometer. 
\end{proof}

\begin{proposition}
\label{pro:envelopingOdometerMetrizable}
    A subodometer $(X,G)$ is metrizable if and only if its enveloping odometer $(\envelopingOdometer{X},G)$ is metrizable. 
\end{proposition}
\begin{proof}
    Clearly, if $(\envelopingOdometer{X},G)$ is metrizable, then so is its factor $(X,G)$. 
    For the converse assume that $(X,G)$ is metrizable. 
    From Theorem \ref{the:characterizationMetrizableSubodometers} we know that $(X,G)$ is generated by a countable scale $S$ and hence also $S_G$ is countable. 
    Since $S_G$ generates $(\envelopingOdometer{X},G)$ we observe from Theorem \ref{the:characterizationMetrizableSubodometers} that $(\envelopingOdometer{X},G)$ is metrizable.     
\end{proof}

We will next present the interplay of the enveloping odometer and the Ellis semigroup $(E(X),G)$ of a minimal action. 
For the definition and various properties of the Ellis semigroup see \cite[Section 3]{auslander1988minimal}. 

\begin{proposition}
\label{pro:EllisSemiGroupAndEnvelopingOdometers}
    Let $(X,G)$ be a minimal action. 
    \begin{itemize}
        \item[(i)] $(X,G)$ is a subodometer if and only if $(E(X),G)$ is an odometer.
        \item[(ii)] If $(X,G)$ is a subodometer, then $(E(X),G)$ and $(\envelopingOdometer{X},G)$ are conjugated. 
    \end{itemize}     
\end{proposition}
\begin{proof} 
(i): 
    Assume that $(E(X),G)$ is an odometer. 
    By \cite[Theorem 3.6]{auslander1988minimal} there exists a factor map $E(X)\to X$ and it follows from Proposition \ref{pro:factorsOfSubodometers} that $(X,G)$ is a subodometer.  

    For the converse assume that $(X,G)$ is a subodometer. 
    Since $(X,G)$ is a minimal equicontinuous action it follows from \cite[Section 3]{auslander1988minimal} that $(E(X),G)$ is a minimal rotation. 
    Consider a factor map $\pi\colon \envelopingOdometer{X}\to X$. 
    From \cite[Theorem 3.7]{auslander1988minimal} we know that there exists a factor map $(E(\envelopingOdometer{X}),G)\to (E(X),G)$. 
    Since $\envelopingOdometer{X}$ is a group rotation, we observe that $(\envelopingOdometer{X},G)$ and $(E(\envelopingOdometer{X}),G)$ are conjugated and hence there exists a factor map $\envelopingOdometer{X}\to E(X)$. 
    Since factors of subodometers are subodometers, it follows from Proposition \ref{pro:factorsOfSubodometers} that $(E(X),G)$ is a subodometer. Since it is a rotation it is an odometer. 

(ii): 
    Let $(X,G)$ be a subodometer. As presented in the proof of (i) there exists a factor map $\envelopingOdometer{X}\to E(X)$. 
    Furthermore, from (i) we know that $E(X)$ is an odometer. Since $(X,G)$ is a factor of $(E(X),G)$ there exists a factor map $E(X)\to \envelopingOdometer{X}$. 
    It follows from the coalescence of minimal equicontinuous actions that $(E(X),G)$ and $(\envelopingOdometer{X},G)$ are conjugated. 
\end{proof}

\bibliographystyle{alpha}
\bibliography{ref}
\end{document}